\documentclass[letterpaper,11pt, leqno]{amsart}

\usepackage[margin=1.2in, marginparwidth=1.5cm, marginparsep=0.5cm]{geometry}
\usepackage{amsmath,amssymb,amsthm,mathtools}
\usepackage{xcolor}
\usepackage{needspace}

\usepackage{tikz}
\usetikzlibrary{decorations.pathreplacing,arrows.meta}

\usepackage[implicit=true,colorlinks=true,
  linkcolor=blue,citecolor=blue,urlcolor=blue]{hyperref}
\hypersetup{
  pdftitle={Lp maximal inequalities for Gaussian subordination
    with applications to Breuer--Major--Donsker principles},
  pdfauthor={Dionysis Milesis and Guangqu Zheng}
}

\allowdisplaybreaks[2]
\theoremstyle{plain}
\newtheorem{theorem}{Theorem}[section]
\newtheorem{lemma}[theorem]{Lemma}
\newtheorem{proposition}[theorem]{Proposition}
\newtheorem{corollary}[theorem]{Corollary}
\theoremstyle{remark}
\newtheorem{remark}[theorem]{Remark}

\newcommand{\E}{\mathbb{E}}
\newcommand{\PP}{\mathbb{P}}
\newcommand{\R}{\mathbb{R}}
\newcommand{\Z}{\mathbb{Z}}

\newcommand{\cF}{\mathcal{F}}
\newcommand{\cN}{\mathcal{N}}
\newcommand{\cW}{\mathcal{W}}
\newcommand{\ind}{\mathbf{1}}

\newcommand{\norm}[1]{\|#1\|}
\newcommand{\al}{\alpha}
\newcommand{\be}{\beta}
\newcommand{\dl}{\delta}
\newcommand{\dr}{\theta}
\newcommand{\eps}{\varepsilon}
\newcommand{\ga}{\gamma}
\newcommand{\ld}{\lambda}
\newcommand{\s}{\sigma}
\newcommand{\wt}{\widehat}

\newcommand{\LRA}{\Longrightarrow}

\newcommand{\cZ}{\mathcal Z}

\DeclareMathOperator{\Var}{Var}

\newcommand{\noi}{\noindent}

\newcommand{\AND}{\quad{\rm and}\quad}

\newcounter{pfstep}
\renewcommand{\thepfstep}{S\arabic{pfstep}}
\newcommand{\pstep}[2]{%
  \par\smallskip\Needspace{3\baselineskip}%
  \refstepcounter{pfstep}\label{#1}%
  \noindent\textbf{(\thepfstep) #2.}\ }

\numberwithin{equation}{section}

 \makeatletter
\@namedef{subjclassname@2020}{\textup{2020 Mathematics Subject
Classification}}
\makeatother


\begin{document}

\title[Maximal inequalities and BMD principles]
{$L^p$ maximal inequalities for Gaussian subordination\\
with applications to Breuer--Major--Donsker principles}

\author[Dionysis Milesis]{Dionysis Milesis}
\address{Department of Mathematics and Statistics, Boston University,
665 Commonwealth Avenue, Boston, MA 02215, USA}
\email{dmilesis@bu.edu}

\author[Guangqu Zheng]{Guangqu Zheng}
\address{Department of Mathematics and Statistics, Boston University,
665 Commonwealth Avenue, Boston, MA 02215, USA}
\email{gzheng90@bu.edu}

\date{\today}

\begin{abstract}
In this paper, we establish an   $L^2$ maximal inequality for
partial sums of Gaussian-subordinated random variables.
More precisely, let $f:\mathbb{R}\to\mathbb{R}$ belong to 
$L^2(\gamma)$ and have Hermite rank at least $d$,
where $\gamma$ denotes the standard Gaussian measure.
If the covariance matrix of the underlying Gaussian family,
with standard normal marginals, satisfies a uniform
$d$th-power row bound $K$, then the $L^2$ norm of the maximal
partial sum is bounded by $C_{d,K}\sqrt n\,\|f\|_{L^2(\gamma)}$.
The constant $C_{d,K}$ is independent of $(f,n)$ and depends
on the covariance structure only through $K$.

The proof combines rank-preserving truncation, separated diagram
estimates, a low--high chaos decomposition of the squared
truncation, and a key adjacent-block maximal lemma.  
 We also give the
corresponding $L^p$ inequalities for $2<p<\infty$, using the
Malliavin--Meyer argument of Nourdin and Nualart
(Probab. Theory Related Fields, 2020) and dyadic summation.
The $L^2$ case is critical for this dyadic argument: the geometric
decay available for $p>2$ disappears at $p=2$, where standard
dyadic summation produces a logarithmic loss.

As a consequence, our  $L^2$ maximal inequality controls Hermite
tails uniformly in the path norm.  In discrete time, this yields
the Breuer--Major--Donsker principle under the sole finite-variance
assumption, removing both the additional $L^{2+}$-integrability
assumption of Nourdin and Nualart
(Probab. Theory Related Fields, 2020) and the prediction-theoretic
or decimation assumptions of Mansanarez, Poly, and Zheng
(arXiv:2607.11469).  In continuous time, we likewise obtain the
Breuer--Major--Donsker principle under finite variance alone,
removing the additional moment assumption of Campese, Nourdin,
and Nualart (Ann. Probab., 2020).  For the nonstationary
self-similar setting considered there, we remove the same extra
moment assumption in the subcritical regime, while a separate
leading-chaos argument yields the corresponding logarithmically
normalized critical limit. 
\end{abstract}

\maketitle

\tableofcontents

\section{Introduction and main results}
\label{intro}

This paper investigates $L^p$ maximal inequalities for partial-sum
processes of dependent random variables obtained by Gaussian
subordination:  $X_i= f(G_i)$, $i\in\mathcal I$, where
$(G_i)_{i\in\mathcal I}$ is a jointly Gaussian family and $f$
is measurable.  Our main contribution is the critical case
$p=2$, where we establish a finite-variance maximal inequality that removes the
logarithmic loss arising in the usual dyadic argument and retains  the
optimal $\sqrt n$ scale.  This estimate controls Hermite tails
uniformly in the path norm and yields the
Breuer--Major--Donsker (BMD) principles under the finite-variance
assumption alone.  In the corresponding stationary BMD results,
this removes the extra Hermite-coefficient assumptions of
Ben Hariz \cite{BH02}, the moment assumptions of Nourdin and
Nualart \cite{NN20} and Campese, Nourdin, and Nualart
\cite{CNN20}, and the prediction-theoretic or decimation
assumptions of Mansanarez, Poly, and the second author
\cite{MPZ26}.  We also treat the self-similar setting of \cite{CNN20},
removing the additional moment assumption there as well. 
 The higher-moment inequalities for $p>2$
complete our treatment of the maximal inequalities in this setting.

We work on a probability space $(\Omega,\cF,\PP)$ and write
$\E$ for expectation.  For $X\in L^p(\Omega)$ with $p\in[1,\infty)$,
we write
$\|X\|_p=(\E|X|^p)^{1/p}$.
Let $\ga$ denote the standard Gaussian measure on $\R$ and write
$\|f\|_p=\|f\|_{L^p(\ga)}$ when there is no ambiguity.
We use $\cZ\sim\cN(0,1)$ and $\cW$ for standard Brownian motion.
The letter $C$ denotes a positive finite constant
whose value may change from line to line; relevant parameter
dependence is specified by subscripts.
For a finite set $I$, we write $|I|$ for its cardinality.
The symbol $\ind_E$ denotes an indicator.

\subsection{\texorpdfstring{$L^p$}{Lp} maximal inequalities
for Gaussian subordination}
\label{max_result}

Maximal inequalities are fundamental tools in probability theory,
with important applications to laws of large numbers, tightness,
and uniform approximation of stochastic processes.
  For independent centered square-integrable random
variables $X_1,\ldots,X_n$, put $S_m=\sum_{k=1}^mX_k$ and
$S_n^*=\max_{0\le m\le n}|S_m|$.  Kolmogorov's maximal inequality
gives the tail bound
\begin{align*}
 \PP(S_n^*>\ld)
 \le\frac{1}{\ld^2}\sum_{k=1}^n\E[X_k^2],\quad \ld>0,
\end{align*}
while Doob's inequality, applied to the partial-sum martingale,
gives the strong second-moment estimate
\begin{align*}
 \E[(S_n^*)^2]\le4\E[S_n^2]
 =4\sum_{k=1}^n\E[X_k^2].
\end{align*}
For Kolmogorov's maximal inequality and Doob's $L^p$ maximal
inequality, see \cite[Theorems~5.3.1 and~9.5.4]{Chu01},
respectively.
See also the survey \cite{KS23} of
K\"uhn and Schilling  
for a  broad overview of
maximal inequalities and their applications.

More generally, for a martingale $(M_k)_{k=0}^n$ with
$M_n\in L^p(\Omega)$, Doob's $L^p$ inequality states that
 $\|\max_{0\le k\le n}|M_k| \|_p
 \le\frac{p}{p-1}\|M_n\|_p$, $1<p<\infty.$
For independent centered $X_k\in L^p(\Omega)$ and $p>2$,
Rosenthal's inequality \cite{Ros70} gives
\begin{align*}
 \left\|\sum_{k=1}^n X_k\right\|_p
 \le C_p\left[
 \left(\sum_{k=1}^n\|X_k\|_2^2\right)^{\frac12}
 +
 \left(\sum_{k=1}^n\|X_k\|_p^p\right)^{\frac1p}
 \right].
\end{align*}
Applying Doob's inequality to the partial-sum martingale then yields
the bound for 
 $\|S_n^*\|_p$.
In particular, for identically distributed summands this is bounded
by $C_p\sqrt n\,\|X_1\|_p$.  Thus finite variance is sufficient
for the natural $\sqrt n$   scale at $p=2$, while higher
maximal moments are controlled by the corresponding moments of the
summands.

For dependent sums, maximal estimates reflect both the moment
assumptions and the dependence structure.
M\'oricz \cite{Mor76} converts uniform block moment bounds with
a strictly superlinear length exponent into maximal bounds of
the same order.  At the second-moment boundary, a direct dyadic
argument based on  $\E|S_I|^2\le C|I|$,
where $S_I = \sum_{i\in I} X_i$,
incurs a logarithmic
loss in the maximal $L^2$ norm.
 The same logarithm 
appears in the orthogonal-series inequalities
of Menchoff and Rademacher \cite{Men23,Rad22}
and   Salem's argument \cite{Sal41}; 
see also  \cite[pp.~193--194]{Zyg59}.

A different approach is the strong $L^2$ projective inequality
of Peligrad and Utev \cite[Proposition~2.3]{PU05}.
For a centered stationary sequence with partial sums $S_m$,
their bound involves the projective quantity
\begin{align*}
 \Delta_r
 =
 \sum_{j=0}^{r-1}
 2^{-j/2}
 \|\E[S_{2^j}  | \mathcal{F}_0]\|_2,
 \quad 2^{r-1}<n\le2^r,
\end{align*}
and gives
$
 \left\|\max_{k\le n}|S_k|\right\|_2
 \le
 3\sqrt{n}  (\|X_1\|_2+\Delta_r ).
$
If one uses the linear block-variance bound
$\|S_m\|_2^2\le Cm$, then
$2^{-j/2}  \|\E[S_{2^j} | \mathcal F_0]\|_2 \leq C^{1/2}$
at every dyadic scale.  Hence
$\Delta_r\lesssim r\asymp\log n$, yielding only
\begin{align*}
 \left\|\max_{k\le n}|S_k|\right\|_2
 \lesssim \sqrt n\,\log(2n).
\end{align*}
One may remove the logarithm-term in  the projective inequality 
only when the conditional block sums satisfy additional summable decay;
such decay does not follow from the covariance row bound  alone.
For a systematic treatment of such methods for
dependent sequences, see Merlev\`ede, Peligrad, and Utev
\cite{MPU19}.

Gaussian subordination offers a particularly concrete description
of dependence through summability or row-wise control of covariance
powers, with the relevant power dictated by the Hermite rank.
Bardet and Surgailis \cite{BS13a} establish moment bounds for
nonstationary Gaussian-subordinated arrays under covariance-power
row conditions.  The Malliavin--Meyer argument of Nourdin and
Nualart \cite{NN20} controls increments under an $L^p(\ga)$
assumption with $p>2$, and 
Ben Hariz and Bouzebda \cite{BHB25}
prove a Rosenthal-type maximal inequality for stationary Gaussian
subordination, with a refined $L^4$ estimate separating second-
and fourth-moment contributions.
These results inspire the moment and
truncation methods used below.

\medskip
\noi
$\bullet$ {\bf Our framework.}
We now specify the framework, with particular emphasis on the
finite-variance endpoint $p=2$.
Since Gaussian subordination is
naturally analyzed through the orthogonal decomposition of
$L^2(\ga)$ into Hermite chaoses, Hermite expansions will play a
fundamental role throughout the paper.  Recall that the Hermite
polynomials are defined by Rodrigues' formula

\noi
\begin{align*}
 H_q(x)=(-1)^q e^{x^2/2}
 \frac{d^q}{dx^q}\bigl(e^{-x^2/2}\bigr),\quad q\ge0.
\end{align*}
Thus $H_0(x)=1$, $H_1(x)=x$, $H_2(x)=x^2-1$, and
$H_{q+1}(x)=xH_q(x)-qH_{q-1}(x)$ for $q\ge1$.
The normalized polynomials
\begin{align}\label{def_hq}
 h_q=\frac{H_q}{\sqrt{q!}},\quad q\ge0,
\end{align}
form an orthonormal basis of $L^2(\ga)$; see
\cite[Section~1.4]{NP12}.  A function $\varphi\in L^2(\ga)$
has Hermite rank $d\ge1$ if
\begin{align}\label{d_phi}
 \varphi=\sum_{q=d}^{\infty}c_qH_q,
 \quad c_d\ne0,
 \quad \|\varphi\|_2^2
 =\sum_{q=d}^{\infty}q!c_q^2<\infty.
\end{align}
In particular, $\varphi$ is centered under $\ga$.
We say that $\varphi$ has Hermite rank at least $d$ if its
coefficients of orders $0,\ldots,d-1$ vanish; this convention
includes the zero function.
For jointly Gaussian $X,Y$ with standard normal marginals and
$\E[XY]=\rho$, we will use the orthogonality relation
\begin{align}\label{OL_pq}
 \E[H_p(X)H_q(Y)]
 =p!\,\rho^p\,\ind_{\{p=q\}};
\end{align}
see \cite[Proposition~2.7.5 and Theorem~2.7.7]{NP12}.

\medskip

We now state our main result, whose proof is given in
Sections~\ref{lp_sec} and \ref{ns_sec}.

\begin{theorem}[\textsf{$L^2$ endpoint and $L^p$ maximal inequalities}]
\label{ns_max}
Fix an integer $d\ge1$.
Let $Z_1,\ldots,Z_n$ be jointly Gaussian with standard normal
marginals, and suppose that
\begin{align}\label{ns_rows}
 \max_{1\le i\le n}
 \sum_{j=1}^n|\E[Z_iZ_j]|^d\le K<\infty.
\end{align}
For every $f\in L^2(\ga)$ with Hermite rank at least $d$,
\begin{align}\label{ns_l2}
 \E\left[\max_{0\le m\le n}
 \left|\sum_{k=1}^m f(Z_k)\right|^2\right]
 \le C_{d,K}n\|f\|_2^2.
\end{align}
More generally, for every $2<p<\infty$ and every
$f\in L^p(\ga)$ with Hermite rank at least $d$,
\begin{align}\label{ns_bdd}
 \left\|\max_{0\le m\le n}
 \left|\sum_{k=1}^m f(Z_k)\right|\right\|_p
 \le C_{p,d,K}\sqrt n\,\|f\|_p.
\end{align}

\end{theorem}

\begin{remark}
\rm
(i) The constant $C_{d, K}$ in  the endpoint estimate
\eqref{ns_l2} is uniform in $f\in L^2(\gamma)$ and
in the   covariance matrix, subject to the Hermite-rank
condition and row-sum bound \eqref{ns_rows}.
The exponent $p=2$ is critical for the usual dyadic maximal
argument.  A block estimate of order $\sqrt{|I|}$ in $L^p$
produces the factors $2^{-j(\frac{1}{2}-\frac{1}{p})}$ at level
$j$.  These form a summable geometric series for $p>2$;
at $p=2$ the decay disappears, so summation over the dyadic
levels produces a logarithmic loss.\footnote{The geometric
series is evaluated in \eqref{lp_geometric} and used in
\eqref{lp_max_proof}.  At $p=2$, its finite counterpart has
$L+1$ terms, where $L$ is of order $\log n$.}
Indeed, from the orthogonality relation \eqref{OL_pq},
we can deduce that
$\|\sum_{i\in I}f(Z_i)\|^2_2
 \le\|f\|_2^2\sum_{i,j\in I}|\E[Z_iZ_j]|^d
 \le K|I|\,\|f\|_2^2$,
and direct dyadic summation gives only
$C\sqrt{Kn}\log(2n)\|f\|_2$ for the maximal $L^2$ norm.
The main estimate \eqref{ns_l2} removes this logarithmic loss
under the sole finite-variance assumption.
It is the key input used for
the BMD principles in Section~\ref{bm_intro}.

\smallskip

(ii) For independent Gaussian coordinates, \eqref{ns_l2} follows
from Doob's inequality.  In the dependent setting, the partial
sums of $f(Z_k)$ are generally not martingales, so neither
Doob's inequality nor the Burkholder--Davis--Gundy inequalities
\cite{BDG72} apply directly; see
\cite[Chapters~II and~IV]{RY99} for standard treatments.
The right-hand side of \eqref{ns_l2} is $n\|f\|_2^2$, not the terminal variance.
For $G\sim\cN(0,1)$, the choice $Z_1=G$, $Z_2=-G$, and
$f(x)=x$ satisfies \eqref{ns_rows} with $d=1$, $K=2$, but the
terminal sum is zero and the maximal partial sum is $|G|$.
Gaussian entropy and concentration inequalities
  give the bound when $f(x)=ax$ and
$d=1$; see, e.g.,  \cite[Chapters~1--2]{AT07}.
For nonlinear $f$, however, the partial-sum process need not
be Gaussian, and square integrability alone does not supply
sub-Gaussian increment control.  For a fixed finite Hermite
expansion, Nelson's hypercontractivity, followed
by a higher-moment maximal argument, gives a bound with a
constant depending on the Hermite cutoff; see
\cite[Section~2.8]{NP12}.  The endpoint \eqref{ns_l2} controls
the full $L^2(\ga)$ class of the prescribed  Hermite rank 
uniformly in that cutoff.

\smallskip

(iii) The resulting $\sqrt n$ scale agrees with that for independent
sums and square-integrable martingales, and is optimal: for
independent standard Gaussian coordinates and $f=h_d$, the
terminal sum already has $L^2$ norm $\sqrt n$.

\end{remark}

The higher-moment regime $p>2$ and the critical endpoint $p=2$
require different arguments.

\medskip
\noindent
$\bullet$ \textbf{The higher-moment case $p>2$.}
The increment argument of Nourdin and Nualart \cite{NN20},
applied under the row condition \eqref{ns_rows}, gives
\begin{align}\label{lp_preview}
 \left\|\sum_{k\in I}f(Z_k)\right\|_p
 \le C_{p,d}\sqrt{K|I|}\,\|f\|_p.
\end{align}
A dyadic decomposition then yields the maximal inequality:
the contribution of level $j$ is bounded by a constant times
$\sqrt n\,2^{-j(\frac{1}{2}-\frac{1}{p})}\|f\|_p$, and the resulting
geometric
series converges.  Section~\ref{lp_sec} gives the
Malliavin--Meyer block estimate and this summation for every
$p\in (2, \infty)$.

\medskip
\noindent
$\bullet$ \textbf{The critical endpoint $p=2$.}
Our proof in Section~\ref{ns_sec} removes the logarithmic loss
by controlling two adjacent blocks simultaneously.
The proof is organized into the following six steps.

\pstep{step_trunc}{Rank-preserving truncation}
The case $f=0$ is immediate.  By homogeneity, we may otherwise
assume $\|f\|_2=1$.  Set
\begin{align}\label{intro_trunc}
 T_n=f\ind_{\{|f|\le\sqrt n\}}
 \AND g_n=T_n-\Pi_{<d}T_n,
\end{align}
where $\Pi_{<d}$ is the orthogonal projection onto
$\operatorname{span}\{h_0,\ldots,h_{d-1}\}$.
Then $g_n$ has Hermite rank at least $d$ and, by
Proposition~\ref{clip_prop}, satisfies $\|g_n\|_2\le1$ and
$\|g_n\|_4^4\le C_dn$.

\pstep{step_diagrams}{Separated diagram estimates}
We first impose the fixed small-correlation bound
\eqref{assump_31} on distinct coordinates; Step~\eqref{step_partition}
removes this restriction.  For products at distinct coordinates,
the common-rank estimates in
Lemma~\ref{diag_lem}\textup{(i)--(ii)} are scalar consequences
of Bardet and Surgailis \cite[Lemma~1]{BS13a}.
We also prove a mixed-rank refinement in
Lemma~\ref{diag_lem}\textup{(iii)}: when one factor belongs to
a fixed chaos of order $1\le q<d$ and the other two have rank
at least $d$, retaining the low-order powers of the correlations improves the
bound from $O(m^2)$ to $O(m^{2-\frac{q}{2d}})$.

\pstep{step_square}{Low--high decomposition of the square}
Repeated indices in the mixed fourth moment introduce $g_n^2$,
which may have nonzero low-order Hermite coefficients.
We therefore write
\begin{align}\label{intro_square}
 g_n^2-\E[g_n(\cZ)^2]
 =\sum_{q=1}^{d-1}B_{q,n}+\mathcal H_n,
 \quad B_{q,n}=\langle g_n^2,h_q\rangle h_q,
\end{align}
where $\mathcal H_n$ has Hermite rank at least $d$.
Lemma~\ref{sq_lem} controls the low-order coefficients by
Gaussian small-set estimates and the high-order part by
orthogonality and Ornstein--Uhlenbeck hypercontractivity.
These bounds are then combined with the separated diagram
estimates from Step~\eqref{step_diagrams}.

\pstep{step_adjacent}{Adjacent blocks and their maximum}
For $S_I(g_n)=\sum_{i\in I}g_n(Z_i)$, the preceding estimates
give, for adjacent disjoint intervals $I,J$,
\begin{align*}
 \E[S_I(g_n)^2S_J(g_n)^2]
 \le C_{d,K}n^{1-\al}(|I|+|J|)^{1+\al},
 \quad \al=\frac{1}{4d}>0.
\end{align*}
This strict gain makes the dyadic argument in
Lemma~\ref{max_lem} summable and yields
\begin{align*}
 \left\|\max_{m\le n}
 \left|\sum_{k=1}^m g_n(Z_k)\right|\right\|_2
 \le C_{d,K}\sqrt n.
\end{align*}

\pstep{step_restore}{Restoration in $L^2$}
It remains to restore the discarded tail $f-T_n$ and the
low-Hermite correction $\Pi_{<d}T_n$.
Lemma~\ref{small_support} controls the former, while
\eqref{ns_lowmax}, together with the bounds on the removed
Hermite coefficients, controls the latter.
Consequently, for $\|f\|_2=1$,
\begin{align*}
 \left\|\max_{0\le m\le n}
 \left|\sum_{k=1}^m(f-g_n)(Z_k)\right|\right\|_2
 \le C_{d,K}\sqrt n.
\end{align*}

\pstep{step_partition}{Finite partition and recombination}
Finally, the row condition \eqref{ns_rows} bounds the number of
coordinates that can be strongly correlated with any fixed coordinate.
This allows us to partition the underlying Gaussian family into
finitely many subfamilies, with the number of classes depending only
on $(d,K)$, such that distinct coordinates within each subfamily
have uniformly small correlations.  Each subfamily therefore falls
within the setting of
Steps~\eqref{step_diagrams}--\eqref{step_restore}.
Applying the separated estimate in the inherited order within
each class and recombining the classes proves \eqref{ns_l2}.

Our strategy is inspired by the rank-preserving truncation used in
\cite[p.~1523]{BHB25} and by the decimation-by-residue-classes
argument used in \cite{BS13b,LPV23}.  This has the same
partition--estimate--recombine structure as the decimation
strategy in \cite{MPZ26};
see  \cite[Remark 3.6]{MPZ26} for the discussion on the decimation technique.

The proof of the $L^2$ endpoint result is given in  Section~\ref{max_pf}.

\subsection{The Breuer--Major--Donsker principle}
\label{bm_intro}

As anticipated, let us now turn to the
Breuer--Major--Donsker principle, or simply the BMD principle.
For independent identically distributed centered random variables
$(X_k)_{k\ge1}$ with variance $\sigma^2<\infty$, the classical
central limit theorem (CLT) asserts that
\begin{align*}
 \frac{1}{\sqrt n}\sum_{k=1}^nX_k
 \LRA\cN(0,\sigma^2).
\end{align*}
The Gaussian law is thus the universal limit for these normalized
sums.  Donsker's invariance principle identifies the corresponding
process limit: the normalized partial-sum process converges to
$\sigma\cW$, where $\cW$ is standard Brownian motion.

The celebrated theorem of Breuer and Major \cite{BM83} extends
CLTs to nonlinear functions of stationary
Gaussian sequences.  Their work was motivated by the non-central
limit theorems of Dobrushin, Major, Rosenblatt, and Taqqu; see
\cite{Taq77,Taq79,DM79,Maj81,Ros81}.  It identifies a summable regime,
determined jointly by the covariance
and the Hermite rank, in which the Gaussian limit persists even though
the underlying Gaussian sequence may exhibit substantial dependence.
 The original proof uses moments and diagram formulae.  Modern proofs
also follow from the fourth moment theorem of Nualart and Peccati
\cite{NP05}, its multivariate extension by Peccati and Tudor
\cite{PT05}, and the chaotic central limit theorem of Hu and
Nualart \cite{HN05}; see \cite[Chapter~7]{NP12} and
\cite{Nou13} for detailed accounts.

Let $\mathbf G=(G_k)_{k\in\Z}$ be a centered stationary Gaussian
sequence with covariance $\E[G_kG_\ell]=\rho_G(k-\ell)$ and
$\rho_G(0)=1$, and let $\varphi\in L^2(\ga)$.
Define
\begin{align}\label{d_V}
 V_n(t)=\frac{1}{\sqrt n}
 \sum_{k=1}^{\lfloor nt\rfloor}\varphi(G_k),
 \quad t\in[0,1].
\end{align}

Now we are ready to state the Breuer--Major theorem.

\begin{theorem}[\textsf{Breuer--Major theorem}
\cite{BM83}]\label{bm_thm}
Assume that $\varphi\in L^2(\ga)$ has Hermite rank $d\ge1$ as in
\eqref{d_phi} and $\rho_G\in\ell^d(\Z)$.  Then the
finite-dimensional distributions of $V_n$ converge to those of
$\s_G\cW$, where
\begin{align}\label{d_var}
 \s_G^2
 =\sum_{q=d}^{\infty}q!c_q^2
   \sum_{k\in\Z}\rho_G(k)^q
 \in[0,\infty).
\end{align}
\end{theorem}

There is a corresponding continuous-time theorem.  Let
$\mathbf B=(B_x)_{x\in\R}$ be a jointly measurable centered
stationary Gaussian process with
\begin{align}\label{c_cov}
 \E[B_xB_y]=\rho_B(x-y),\quad \rho_B(0)=1,
\end{align}
and put
\begin{align}\label{c_V}
 V_T(t)=\frac{1}{\sqrt T}
 \int_0^{Tt}\varphi(B_x)\,dx,\quad t\in[0,1].
\end{align}
By joint measurability and $\varphi\in L^1(\ga)$,
these integral paths are almost surely continuous.  
Joint measurability and stationarity also imply continuity of $\rho_B$
(see Section~\ref{cont_pf}).
If $\varphi$ has Hermite rank $d\ge1$ and
\begin{align}\label{c_bm}
 \int_\R|\rho_B(x)|^d\,dx<\infty,
\end{align}
then the finite-dimensional distributions of $V_T$ converge, as
$T\to\infty$, to those of $\s_{\varphi,B}\cW$, where
\begin{align}\label{c_var}
 \s_{\varphi,B}^2
 =\sum_{q=d}^{\infty}q!c_q^2
   \int_\R\rho_B(x)^q\,dx\in[0,\infty).
\end{align}

\noi
Continuous-time CLTs of this type go back to Maruyama
\cite{Mar76} and Chambers and Slud \cite{CS89}; the formulation
above appears in Ben Hariz \cite{BH02}.

The missing ingredient for functional convergence is
therefore tightness.
The terminology \emph{Breuer--Major--Donsker principle},
introduced in \cite{MPZ26}, reflects the combination of the
Breuer--Major finite-dimensional limits with Donsker-type
convergence of the entire partial-sum path.  We use
$\mathbf D([0,1])$ for the space of real-valued c\`adl\`ag paths
with the Skorokhod $J_1$ topology and $\mathbf C([0,1])$ for the
space of continuous paths with the uniform norm
$\|x\|_\infty=\sup_{t\in[0,1]}|x(t)|$.\footnote{We formulate
all functional limit theorems on $[0,1]$ for simplicity.
The same arguments apply on every fixed interval $[0,T]$,
$T<\infty$, and hence yield the corresponding convergence in
$\mathbf D(\R_+;\R)$ with the local Skorokhod $J_1$ topology,
or in $\mathbf C(\R_+;\R)$ with the topology of uniform
convergence on compact intervals.  The same convention applies
to vector-valued limits.}
 Unless otherwise specified,
path norms are uniform norms on $[0,1]$.
A family in $\mathbf D([0,1])$ is $C$-tight if it is tight and
every weak subsequential limit has continuous paths; see
\cite[Chapter~3]{Bil99} for more details.

Earlier functional results of Chambers and Slud \cite{CS89}
and Ben Hariz \cite{BH02} impose stronger summability conditions
on the Hermite coefficients.  For example, one continuous-time
condition in \cite{BH02}, in our normalization, is that for some
$R>1$,
\begin{align}\label{BH_cond}
 \sum_{q=d}^{\infty}|c_q|\sqrt{q!}\,
 R^q\left(\int_\R|\rho_B(x)|^q\,dx\right)^{\frac{1}{2}}<\infty.
\end{align}
Nourdin and Nualart \cite[Theorem~1.1]{NN20} establish what they
call the \emph{functional Breuer--Major theorem} under the
additional assumption $\varphi\in L^p(\ga)$ for some $p>2$.
Their proof uses Meyer's inequalities to control increments and
obtain tightness.  Campese, Nourdin, and Nualart
\cite[Theorems~1.1--1.2]{CNN20} obtain the continuous-time
counterpart under the same moment condition and extend it to a
class of nonstationary self-similar Gaussian processes.

Mansanarez, Poly, and the second author \cite{MPZ26} address
the finite-variance setting through Gaussian prediction theory.
Their non-deterministic case assumes a positive one-step
prediction error,
$\|G_0-\E[G_0\mid\sigma(G_j:j\le-1)]\|_2^2>0$.
They obtain the BMD principle in this case and, under an
additional decimation assumption, in the deterministic case;
see \cite[Theorems~1.3--1.4]{MPZ26}.
The latter result includes $\rho_G\in\ell^1(\Z)$.
Our maximal estimate \eqref{ns_l2} removes these restrictions
through uniform path-norm approximation and yields the
following stationary results.

\medskip
\noindent
$\bullet$ \textbf{Stationary BMD principles.}
Define the polygonal interpolation of $V_n$ by
\begin{align}\label{hat_V}
 \wt{V_n}(t)
 =V_n(t)+\frac{nt-\lfloor nt\rfloor}{\sqrt n}
   \varphi(G_{\lfloor nt\rfloor+1}),
 \quad 0\le t\le1.
\end{align}

\begin{corollary}[\textsf{Finite-variance discrete BMD principle}]
\label{disc_thm}
Let $\mathbf G$ be a centered stationary Gaussian sequence with
covariance function $\rho_G$, and let $V_n$ be as in \eqref{d_V},
with $\varphi\in L^2(\ga)$ having Hermite rank $d\ge1$ as in
\eqref{d_phi}.  Assume that $\rho_G(0)=1$ and
\begin{align}\label{d_bm}
 \sum_{k\in\Z}|\rho_G(k)|^d<\infty.
\end{align}
Then, as $n\to+\infty$,   $V_n\LRA\s_G\cW$ in  $\mathbf D([0,1])$,
where $\s_G^2$ is given by \eqref{d_var}.  The family $(V_n)$ is
$C$-tight.  Moreover, with $\wt{V_n}$ as in \eqref{hat_V},
 $\wt{V_n}\LRA\s_G\cW$
 in $\mathbf C([0,1])$.

\end{corollary}

\begin{corollary}[\textsf{Finite-variance continuous-time BMD principle}]
\label{cont_thm}
Let $\mathbf B$ be a jointly measurable, centered Gaussian process with
covariance structure \eqref{c_cov} and \eqref{c_bm}. 
Let $\varphi\in L^2(\ga)$ have Hermite rank $d\ge1$ as in
\eqref{d_phi}.
Then, as $T\to\infty$,
$ V_T\LRA\s_{\varphi,B}\cW$
in $\mathbf C([0,1]),$
where $\s_{\varphi,B}^2$ is given by \eqref{c_var}.
\end{corollary}

These corollaries establish the functional limits under the same
finite-variance and covariance-summability conditions as the
Breuer--Major finite-dimensional theorems.  In particular, they
remove the additional moment assumptions in \cite{NN20,CNN20}
and the prediction-theoretic assumption in \cite{MPZ26}.

Let us briefly sketch our proof strategy for the BMD principles.
Let $\varphi_M=\sum_{q=d}^M c_qH_q$.  For fixed $M$, each
increment of the associated process belongs to a finite sum of
Wiener chaoses.  The covariance estimate and hypercontractivity
give a moment bound of order strictly greater than two, and hence
tightness.  Combining this with the finite-dimensional
Breuer--Major theorem gives functional convergence for
$\varphi_M$.  In discrete time, this finite-expansion convergence
also follows from \cite{NN20}; here we use the direct
hypercontractive argument.

For a general $\varphi\in L^2(\ga)$, the remaining task is to
control $\varphi-\varphi_M$ uniformly in the path norm.  This is
exactly where the endpoint $p=2$ of Theorem~\ref{ns_max} enters.
We record the approximation estimate for a general Gaussian
vector, so that the same argument also applies to the
nonstationary processes in Section~\ref{ss_model}.
For a Gaussian vector satisfying \eqref{ns_rows},
we define
\begin{align*}
 \mathcal S_n^f(t)
 =\frac{1}{\sqrt n}\sum_{k=1}^{\lfloor nt\rfloor}f(Z_k),
 \quad 0\le t\le1.
\end{align*}
The endpoint gives the following uniform approximation.

\begin{proposition}[\textsf{Uniform approximation in path norm}]
\label{path_cor}
If $f,g\in L^2(\ga)$ and $f-g$ has Hermite rank at least $d$,
then, under \eqref{ns_rows},
\begin{align}\label{path_l2}
 \E \big[   \|\mathcal S_n^f-\mathcal S_n^g\|_\infty^2 \big]
 \le C_{d,K}\|f-g\|_2^2.
\end{align}
The constant is uniform over all $n$ and all covariance matrices
satisfying \eqref{ns_rows}.  In particular, for
$\varphi_M=\sum_{q=d}^M c_qH_q$,
\begin{align}\label{path_tail}
 \sup_n\E  \big[   \|\mathcal S_n^\varphi
                  -\mathcal S_n^{\varphi_M}\|_\infty^2 \big]
 \le C_{d,K}\sum_{q>M}q!c_q^2\xrightarrow{M\to+\infty}0.
\end{align}

\noi
The same estimates hold for the polygonal paths
$\widehat{\mathcal S}_n^f$, obtained by linear interpolation of
$\mathcal S_n^f$ at $0, 1/n, 2/n,\ldots,1$.
\end{proposition}

\begin{proof}
Applying   \eqref{ns_l2} to $f-g$  yields  \eqref{path_l2} immediately.  
The orthogonality relation \eqref{OL_pq}
then implies  \eqref{path_tail}.
On each grid interval,
the difference of the polygonal paths is affine, so its
absolute maximum is attained at an endpoint.  Thus its
supremum equals that of the difference of the step paths.
Hence, the proof is completed.
\end{proof}

For the stationary sequence in Corollary~\ref{disc_thm},
the row constant is
$K=\sum_{k\in\Z}|\rho_G(k)|^d$, and \eqref{path_tail} becomes
\begin{align}\label{bmd_tail_intro}
 \sup_n\E  \big[   \|V_n^\varphi-V_n^{\varphi_M}\|_\infty^2 \big]
 \le C_{d,K}\|\varphi-\varphi_M\|_2^2\longrightarrow0.
\end{align}
The limiting variances converge by the same covariance bound.
Uniform approximation therefore transfers functional convergence
from $V_n^{\varphi_M}$ to $V_n^\varphi$ for
$\varphi\in L^2(\ga)$.  
See Section~\ref{disc_pf} for the details.

\medskip

For continuous-time integrals in Corollary~\ref{cont_thm},
we transfer the maximal inequality by integrating over shifted
lattice fibers.  In the stationary setting, Lemma~\ref{lat_lem}
allows us to choose a lattice spacing for which the sampled
covariance is $\ell^d$-summable.  The following proposition gives
the corresponding $L^2$ path estimate for a jointly measurable
Gaussian process whose fibers satisfy a uniform row bound.

\begin{proposition}[\textsf{Uniform integral approximation}]
\label{ns_int}
Fix an integer $d\ge1$.  Let $Y=(Y(s))_{s\ge0}$ be a jointly
measurable Gaussian process with standard normal marginals.
Suppose that there is some $K<\infty$ such that
\begin{align}\label{ns_fib}
 \sup_{u\in[0,1]}\sup_{i\in\{0,1,\ldots\}}
 \sum_{j=0}^{\infty}
 \big|\E[Y(1+u+i)Y(1+u+j)]\big|^d\le K.
\end{align}
For $g\in L^2(\ga)$ with Hermite rank at least $d$, define
\begin{align}\label{int_proc}
 F_\eps^g(t)=\sqrt\eps\int_0^{t/\eps}g(Y(s))\,ds,
 \quad t\ge0.
\end{align}
These processes have almost surely continuous paths and satisfy
\begin{align}\label{ns_intbd}
 \sup_{0<\eps\le1}
 \E\left[\sup_{t\in[0,1]}|F_\eps^g(t)|^2\right]
 \le C_{d,K}\|g\|_2^2.
\end{align}
For $f,g\in L^2(\ga)$ such that $f-g$ has Hermite rank at least
$d$, the same estimate holds for $F_\eps^f-F_\eps^g$, with
$\|f-g\|_2^2$ on the right-hand side.
\end{proposition}

The proof of Proposition~\ref{ns_int} is given in
Section \ref{cont_pf}. In the stationary setting,
Lemma~\ref{lat_lem}, applied to $|\rho_B|^d$, yields a spacing
$a>0$ such that
\begin{align*}
 \sum_{k\in\Z}|\rho_B(ak)|^d<\infty.
\end{align*}
The rescaled process $Y(s)=B_{as}$ then satisfies
\eqref{ns_fib}, so Proposition~\ref{ns_int} provides uniform
control of Hermite tails in the path norm.
As in the proof of Corollary~\ref{disc_thm}, we first establish
tightness for finite Hermite expansions by hypercontractivity and
then pass to a general $\varphi\in L^2(\ga)$ by uniform
approximation.  Section~\ref{cont_pf} contains the lattice
selection and the complete argument.

\subsection{Self-similar Gaussian processes}
\label{ss_model}

We next retain the self-similar covariance framework of
Campese, Nourdin, and Nualart \cite{CNN20} and remove the additional
moment assumption in their functional CLTs.  The covariance
hypotheses originate in the nonstationary increment CLT of Harnett
and Nualart \cite{HN18}.  Campese, Nourdin, and Nualart extended this
analysis to continuous integral functionals and proved both
finite-dimensional limits under $L^2(\ga)$ and functional limits
under $L^p(\ga)$, $p>2$; see \cite[Theorem~1.2]{CNN20}.

Let $X=(X(t))_{t\ge0}$ be a centered, jointly measurable Gaussian
process, self-similar of order $\be\in(0,1)$, with $\|X(1)\|_2>0$.
Its covariance has the form
\begin{align}\label{ss_cov}
 \E[X(s)X(t)]=s^{2\be}\phi(t/s),
 \quad 0<s\le t.
\end{align}

\noi
We use $\phi$ for the covariance function and $\varphi$ for the
function applied to normalized increments.  
The following assumptions are taken from  \cite[p.~152]{CNN20}.

\smallskip

\textup{\bf (H.1)} There exist $\al\in(0,2\be]$ and $\ld>0$ such that
$\phi(x)=-\ld(x-1)^\al+\psi(x)$ for $x\ge1$, where $\psi$ is
twice differentiable on an open set containing $[1,\infty)$ and,
for a constant $C>0$ and $x>1$,
\begin{align}\label{ss_h1}
 |\psi'(x)|\le Cx^{\al-1},
 \quad |\psi''(x)|\le Cx^{-1}(x-1)^{\al-1}.
\end{align}
If $\al\ge1$, we also require $\psi'(1)=\be\psi(1)$.

\smallskip

\textup{\bf (H.2)} There exist $c>1$, $\nu\in(1,2]$, and $C>0$ such
that, for $x\ge c$,
\begin{align}\label{ss_h2}
 |\phi'(x)|&\le C
 \begin{cases}
 x^{-\nu},&\al<1,\\
 x^{\al-2},&\al\ge1,
 \end{cases}
 &|\phi''(x)|&\le C
 \begin{cases}
 x^{-\nu-1},&\al<1,\\
 x^{\al-3},&\al\ge1.
 \end{cases}
\end{align}
Here $\al$ describes the local increment variance, whereas $\nu$
controls correlations between increments at widely separated times.
The parameter $\al$ need not equal $2\be$.

For $\eps>0$, define the standardized increments
\begin{align}\label{ss_inc}
 \Delta_\eps X(s)=X(s+\eps)-X(s),
 \quad
 Y_\eps(s)=\frac{\Delta_\eps X(s)}{\|\Delta_\eps X(s)\|_2},
 \quad
 Y(s)=Y_1(s),
\end{align}
and the continuous integral processes
\begin{align}\label{ss_proc}
 \widetilde F_\eps^\varphi(t)
 =\frac{1}{\sqrt\eps}\int_0^t\varphi(Y_\eps(s))\,ds,
 \quad F_\eps^\varphi(t)
 =\sqrt\eps\int_0^{t/\eps}\varphi(Y(s))\,ds.
\end{align}
The denominators are positive, the integral paths are almost surely
continuous, and self-similarity gives
\begin{align}\label{ss_law}
 \widetilde F_\eps^\varphi
 \overset{\mathrm{law}}=F_\eps^\varphi
 \quad\text{in }\mathbf{C}([0,1]);
\end{align}
the verification, valid for   $\varphi\in L^2(\ga)$, is given
at the beginning of Section~\ref{selfsim}.  The asymptotic variance
is expressed through
\begin{align}\label{a_def}
 a_\al(h)=\frac{1}{2}\bigl(
 |h+1|^\al+|h-1|^\al-2|h|^\al\bigr),\quad h\in\R,
\end{align}
the covariance of unit increments of fractional Brownian motion
with Hurst index $\frac{\al}{2}$.

\begin{corollary}[\textsf{Subcritical self-similar  FCLT}]\label{ss_thm}
Let $X$ be the self-similar Gaussian process as above, satisfying
\textup{\bf(H.1)}--\textup{\bf(H.2)}, and let $\varphi\in L^2(\ga)$ have
Hermite rank $d\geq 1$
as in
\eqref{d_phi}.  Suppose that
$\al<2- \frac{1}{d}$.  Then
\begin{align}\label{ss_fclt}
 \widetilde F_\eps^\varphi\LRA\s_{\varphi,\al}\cW
 \quad\text{in }\mathbf{C}([0,1]),\quad \eps\downarrow0,
\end{align}
where $\cW$ is standard Brownian motion and
\begin{align}\label{ss_var}
 \s_{\varphi,\al}^2
 =\sum_{q=d}^{\infty}q!c_q^2
   \int_\R a_\al(h)^q\,dh\in[0,\infty).
\end{align}
The same conclusion holds for $F_\eps^\varphi$.
\end{corollary}

Thus the covariance hypotheses, subcritical threshold, and limiting
variance are the same as in \cite[Theorem~1.2\textup{(1)}]{CNN20};
only the extra moment assumption for convergence in $\mathbf{C}([0,1])$ is
removed.  In the boundary case $\al=2- \frac{1}{d}$, Theorem~\ref{ss_crit}
gives the corresponding finite-variance limit after division by
$\sqrt{\log(1/\eps)}$, with variance \eqref{crit_var} determined
by the leading Hermite coefficient.

The same approximation argument applies here once the
nonstationary covariance is controlled.  Combining the covariance
estimates of \cite[Lemmas~4.1--4.3]{CNN20}, Lemma~\ref{ss_env}
yields, under \textbf{(H.1)}--\textbf{(H.2)} and
$\al<2-\frac{1}{d}$,
\begin{align}\label{ss_preview}
 |\E[Y(s)Y(t)]|
 \le C(1+|t-s|)^{-\kappa},
 \quad s,t\ge1,
 \quad \kappa d>1,
\end{align}
where
 $\kappa=  \min\{2-\al,\nu\}$ if $\al<1$
 and $\kappa =  2-\al$ if $\al\ge1$.
 This envelope gives uniformly bounded
covariance-power row sums on every unit-lattice fiber, so that
Proposition~\ref{ns_int} can be applied.
Finite-dimensional convergence for the truncated Hermite expansion
follows from \cite[Theorem~1.2\textup{(1)}]{CNN20}, while tightness is
obtained from a direct application of hypercontractivity.
At the critical threshold,
we isolate the leading chaos and use the same maximal inequality
to remove all higher chaoses after logarithmic normalization.
The details are given in Section~\ref{selfsim}.

The critical limit requires an additional argument for the leading
Hermite chaos.  We state it separately below; its covariance and
logarithmic normalization are proved in Section~\ref{ss_boundary}.

\begin{theorem}[\textsf{Critical self-similar FCLT}]\label{ss_crit}
Let $X$ satisfy \textup{\bf(H.1)}--\textup{\bf(H.2)} as above,
let $\varphi\in L^2(\ga)$ have expansion \eqref{d_phi}, and suppose that
$\al=2- \frac{1}{d}$.  Then
\begin{align}\label{crit_fclt}
 \frac{\widetilde F_\eps^\varphi}{\sqrt{\log(1/\eps)}}
 \LRA\s_*\cW
 \quad\text{in }\mathbf{C}([0,1]),\quad \eps\downarrow0,
\end{align}
where
\begin{align}\label{crit_var}
 \s_*^2=2d!c_d^2
 \left(\frac{\al(\al-1)}{2}\right)^d =   2d!c_d^2
 \left(\frac{(2d-1)(d-1)}{2d^2}\right)^d.
\end{align}
The same conclusion holds for $F_\eps^\varphi$.
When $d=1$, the limit is the zero process.
\end{theorem}

At the end of this Introduction, we provide further remarks and related results.

\begin{remark}\rm

(i) The counterexamples of Chambers and Slud
\cite[pp.~328--329]{CS89} concern Gaussian shift functionals
that may depend on the entire process through spectral
kernels.  In the relevant example, Brownian finite-dimensional
convergence already fails because the normalized variance is
not asymptotically linear in time.  This differs from the
one-coordinate Breuer--Major setting considered here; see also
\cite[Remark~1.6(iv)]{MPZ26}.

\smallskip
(ii) Benning and Nourdin \cite[Theorem~2.8]{BN26} obtain
rough H\"older convergence of the canonical lifts of scalar
Gaussian-subordinated walks under $\rho_G\in\ell^d(\Z)$ and
$\varphi\in L^p(\ga)$ with $p>6$.
Their results and the rough functional theorem of Altman,
Klose, and Perkowski \cite{AKP26} complement our finite-variance
limits in the classical $J_1$ and uniform topologies.

\end{remark}

\medskip
\noi
$\bullet$ \textbf{Organization.}
Section~\ref{tech} collects the auxiliary estimates for the $L^2$
endpoint: Gaussian estimates, separated diagram bounds, and an
adjacent-block maximal lemma.  In Section~\ref{lp_sec}, we prove
Theorem~\ref{ns_max} for $p>2$ using the Malliavin--Meyer block
bound and dyadic summation.  Section~\ref{ns_sec} treats the critical
case $p=2$ by rank-preserving truncation, adjacent-block control,
and $L^2$ restoration.
Section~\ref{stat_pf} contains the proofs of the stationary discrete-
and continuous-time BMD principles.  Section~\ref{selfsim}
verifies the covariance
bounds for the self-similar model and proves the corresponding
subcritical and critical limits.  Finally, Appendix~\ref{app:aux}
contains the proofs of the elementary Gaussian estimates.

\section{Auxiliary estimates for the \texorpdfstring{$L^2$}{L2} endpoint}
\label{tech}

The estimates below are used in the endpoint proof of
Theorem~\ref{ns_max}.  The higher-moment argument is given
separately in Section~\ref{lp_sec}.
Lemma~\ref{rearr_lem} controls the low Hermite coefficients created
by truncation and squaring.  Lemma~\ref{diag_lem} bounds the
separated Gaussian products in the adjacent-block estimate of
Proposition~\ref{adj_prop}.  Finally, Lemma~\ref{max_lem} converts
that estimate into maximal second-moment control.

\subsection{Gaussian estimates}
\label{rearr}

We record two elementary Gaussian estimates below.

\begin{lemma}
\label{rearr_lem}
For each fixed integer  $q\ge0$, there is a constant $C_q$ such that the
following
statements hold.

\smallskip
\textup{(i)} If $A\subset\R$ is measurable and
$p=\ga(A)\in (0, \frac{1}{2}]$, then
\begin{align}\label{set_rearr}
  \int_A h_q(x)^2\,\ga(dx)
  \le C_qp\bigl(1+\log(1/p)\bigr)^q.
\end{align}

\textup{(ii)} Suppose  $U$ is a nonnegative random variable
with  $U\le M$ and $\E[U]\le C_0$
for some deterministic constants $C_0$ and $M\geq 2$.
Then,
\begin{align}\label{wt_rearr}
  \E[U|h_q(\cZ)|]
  \le C_{q,C_0}(1+\log M)^{\frac{q}{2}},
\end{align}
where $\cZ\sim\cN(0,1)$.
\end{lemma}

Part~\textup{(i)} controls the small-set integrals in the coefficient
estimate \eqref{beta_bd}; part~\textup{(ii)} controls the low-order
projections of $T_n^2$ in Lemma~\ref{sq_lem}.  Their logarithmic
growth is needed when the truncation level tends to infinity.
The proof of Lemma~\ref{rearr_lem} is deferred to Appendix~\ref{app:aux}.

\subsection{Separated diagram estimates}
\label{diag}

We supply the separated estimates used in Step~\eqref{step_diagrams}.
The common-rank estimates below are scalar specializations of
Bardet and Surgailis \cite[Lemma~1]{BS13a}, whose result already
allows nonstationary Gaussian arrays and varying functions,
and thus extends the diagram moment bound of Taqqu
\cite[Lemma~4.5]{Taq77}.  Here a diagram records the pairings of
Gaussian factors in the Hermite expansion of a product.

\begin{lemma}[\textsf{Separated diagram bounds}]\label{diag_lem}
Let $\mathcal{I}$ be a countable index set,
and let
 $(Z_j)_{j\in\mathcal I}$ be a   centered Gaussian
family with standard normal marginals.
Define
$r(i,j)=\E[Z_iZ_j]$ for $i,j\in\mathcal{I}$.
 Suppose that there is some $\eps\in(0,\frac{1}{3})$ such that
\begin{align}\label{small_c1}
 \sup_{i\ne j}|r(i,j)|\le\eps ,
\end{align}
and
\begin{align}\label{small_c2}
   K_d:=\sup_{i\in\mathcal I}
       \sum_{j\in\mathcal I}|r(i,j)|^d<\infty.
\end{align}
Given $A\subset\mathcal I$  with cardinality
$|A|  \leq m$,
the following statements hold, with
a constant $C$ depending only on $d$, $\eps$, and $K_d$,
where
a prime on a multiple sum means that all
displayed indices are pairwise distinct.

\smallskip
\textup{(i)} If $f_1, f_2, f_3, f_4\in L^2(\ga)$ have Hermite rank at
least
$d$, then
\begin{align}\label{diag4}
  \sum_{t_1,\ldots,t_4\in A}'
  \left|\E\left[\prod_{i=1}^4f_i(Z_{t_i})\right]\right|
  \le Cm^2\prod_{i=1}^4\norm{f_i}_2.
\end{align}

\textup{(ii)} If $f_0,f_1,f_2\in L^2(\ga)$ all have Hermite rank at least
$d$, then
\begin{align}\label{diag3}
 \sum_{x,y,z\in A}'
 |\E[f_0(Z_x)f_1(Z_y)f_2(Z_z)]|
 \le Cm^{\frac{3}{2}}\prod_{i=0}^2\norm{f_i}_2.
\end{align}

\textup{(iii)}    Let $1\le s<d$ and $B=bh_s$ for some $b\in\R$,
where $h_s$ is given as in \eqref{def_hq}.
Let $f_1,f_2\in L^2(\ga)$ have
Hermite rank at least $d$.  Then
\begin{align}\label{diag_mix}
 \sum_{x,y,z\in A}'
 |\E[B(Z_x)f_1(Z_y)f_2(Z_z)]|
 \le C |b|    m^{2-\frac{s}{2d}}\norm{f_1}_2\norm{f_2}_2.
\end{align}

\end{lemma}

Part~\textup{(i)} is used for four distinct indices in
Proposition~\ref{adj_prop}, and part~\textup{(ii)} for the high-rank
term with one repeated index in \eqref{one_pair}.
The mixed-rank estimate in part~\textup{(iii)}  handles
the low Hermite projection at that repeated index.

\begin{proof}[Proof of Lemma~\ref{diag_lem}]
For parts~\textup{(i)} and \textup{(ii)}, we can apply
\cite[Lemma~1]{BS13a} with scalar coordinates, common rank threshold $d$,
and all $p$ functions subject to that threshold, where $p=4$ and $p=3$,
respectively.  That result applies to nonstationary Gaussian arrays.  In
its notation, $\nu=1$, $\alpha=p$, and
\begin{align*}
 Q_A:=\max_{t\in A}\sum_{s\in A\setminus\{t\}}|r(t,s)|^d\le K_d
\end{align*}
with $K_d$ as in \eqref{small_c2}.
Its small-correlation hypothesis is $\eps<\frac{1}{p-1}$, which is
satisfied by
$\eps\in(0,\frac{1}{3})$.
As a consequence, we have
\begin{align*}
 \sum_{t_1,\ldots,t_p\in A}'
 \left|\E\prod_{i=1}^pf_i(Z_{t_i})\right|
 \le C|A|^{\frac{p}{2}}Q_A^{\frac{p}{2}}\prod_{i=1}^p\|f_i\|_2.
\end{align*}
This gives \eqref{diag4} and \eqref{diag3}.

\medskip

We next prove part~\textup{(iii)}.  We first assume that $f_1$ and
$f_2$ have finite Hermite expansions:
\begin{align*}
 f_1=\sum_{q\ge d}u_qh_q,\quad
 f_2=\sum_{\ell\ge d}v_\ell h_\ell .
\end{align*}

\noi
Recall that
\begin{align*}
 \|B\|_2=|b|,\quad
 \|f_1\|_2^2=\sum_{q\ge d}u_q^2,
 \AND
 \|f_2\|_2^2=\sum_{\ell\ge d}v_\ell^2.
\end{align*}

Let us first represent the Gaussian family as
$
 Z_j= \mathcal{X} (e_j)
 $,
 where $\mathcal{X} $ is an isonormal Gaussian process over a real separable Hilbert
 space
$\mathfrak H$ and
the 
$\{e_j\}$ are unit vectors such that
 $\langle e_i,e_j\rangle_{\mathfrak H}=r(i,j)$;
see \cite[Proposition 7.2.3]{NP12}.  Since
$
 h_q(Z_j)=\frac{1}{\sqrt{q!}}I_q(e_j^{\otimes q}),
$
we deduce from the product formula \cite[Theorem 2.7.10]{NP12} that

\noi
\begin{align*}
 h_q(Z_y)h_\ell(Z_z)
 &=
 \frac{1}{\sqrt{q!\ell!}}
 \sum_{c=0}^{q\wedge\ell}
 c!\binom qc\binom\ell c
 r(y,z)^c
 I_{q+\ell-2c}\big( e_y^{\otimes(q-c)}  \widetilde{\otimes}
 e_z^{\otimes(\ell-c)}  \big),
\end{align*}

\noi
where  $e_y^{\otimes(q-c)}  \widetilde{\otimes}  e_z^{\otimes(\ell-c)} $
is the symmetrization of $e_y^{\otimes(q-c)}  \otimes
e_z^{\otimes(\ell-c)}$.
Then, we deduce from  the orthogonality relation
\cite[Proposition 2.7.5]{NP12} that
 only the value of $c$
satisfying
$q+\ell-2c=s$
can contribute to $\E[ h_s(Z_x) h_q(Z_y)h_\ell(Z_z) ]$.
Put
 $\al=q-c$ and $\be=\ell-c.$
Then,
$q+\ell-2c=s$ is equivalent to
\begin{align}\label{deg_rel}
 \al+\be=s,\quad
 q=\al+c,\quad
 \ell=\be+c.
\end{align}
In particular, since $q,\ell\ge d$,
we have
\begin{align}\label{c_low}
 c\ge d-\min(\al,\be).
 \end{align}
 Since
$ \langle
 e_y^{\otimes\al}\widetilde{\otimes}
 e_z^{\otimes\be},
 e_x^{\otimes s}
  \rangle_{\mathfrak H^{\otimes s}}
 = \langle
 e_y^{\otimes\al} \otimes
 e_z^{\otimes\be},
 e_x^{\otimes s}
  \rangle_{\mathfrak H^{\otimes s}}
 =
 r(x,y)^\al r(x,z)^\be$ for $\al + \be =s$,
 we have
\begin{align}\label{3diag_exact}
\begin{aligned}
 \E[h_s(Z_x)h_{\al+c}(Z_y)h_{\be+c}(Z_z)]
& = \frac{c! \binom{\al+c}{c} \binom{\be+c}{c} s!}{ \sqrt{s! (\al+c)!
(\be+c)! }}
 r(x,y)^\al r(x,z)^\be r(y,z)^c
\\
&= \frac{\sqrt{s!(\al+c)!(\be+c)!}}{\al!\be!c!}\,
 r(x,y)^\al r(x,z)^\be r(y,z)^c.
\end{aligned}
\end{align}
Again by  $\al+\be=s$, the factorial coefficient in
\eqref{3diag_exact} satisfies
\begin{align*}
 \frac{\sqrt{s!(\al+c)!(\be+c)!}}{\al!\be!c!}
 =  \frac{\sqrt{s!}}{\al!\be!}
 \left(    \frac{(\al+c)!}{c!}   \frac{(\be+c)!}{c!}
 \right)^{\frac{1}{2}}
 \leq C_s(1+c)^{\frac{s}{2}}.
\end{align*}
Therefore,
\begin{align}\label{wick3}
 \big|   \E[h_s(Z_x)h_{\al+c}(Z_y)h_{\be+c}(Z_z)]  \big|
\leq   C_s(1+c)^{s/2}
 |r(x,y)|^\al
 |r(x,z)|^\be
 |r(y,z)|^c.
\end{align}

We next sum the correlation factors over the distinct indices.
Since $\al+\be=s$, at least one of $\al,\be$ is at least $s/2$.
Without loss of generality, let us assume
$
 \be\geq  s/2,
$
while the case $\al\geq  s/2$ can be dealt with in the same way.
With $\be\geq s/2$,
we have
 $\be\ge\al$, and thus
$c\ge d-\al$ by \eqref{c_low}.
 For fixed distinct $x,z$,
 it follows from
\eqref{small_c1}
that

\noi
\begin{align*}
\sum_{y\in A\setminus\{x,z\}}
 |r(x,y)|^\al |r(y,z)|^c
\leq
 \eps^{\al+c-d}
 \sum_{y\in A\setminus\{x,z\}}
 |r(x,y)|^\al |r(y,z)|^{d-\al}.
\end{align*}
If $1\le\al<d$, we can deduce from H\"older's inequality
that

\noi
\begin{align*}
 &\sum_y
 |r(x,y)|^\al |r(y,z)|^{d-\al}
\leq
 \left(\sum_y|r(x,y)|^d\right)^{\frac{\al}{d}}
 \left(\sum_y|r(y,z)|^d\right)^{\frac{d-\al}{d}}\le K_d,
\end{align*}
with $K_d$ as in \eqref{small_c2}.
The same bound is also immediate for $\al=0$.  Thus,
\begin{align}\label{conv_bd}
 \sup_{x\ne z}
 \sum_{y\in A\setminus\{x,z\}}
 |r(x,y)|^\al |r(y,z)|^c
 \le K_d\eps^{\al+c-d}.
\end{align}

We now sum the remaining factor $|r(x,z)|^\be$.  By H\"older's inequality,
we have
\begin{align}
 \sum_{x,z\in A}
 |r(x,z)|^\be \ind_{\{x\ne z\}}
& \le
 |A|^{2-2\frac{\be}{d}}
 \Bigg(
   \sum_{x,z\in A }
   |r(x,z)|^d\ind_{\{x\ne z\}}
 \Bigg)^{\frac{\be}{d}} \notag \\
& \le
 |A|^{2-2\frac{\be}{d}}(|A|K_d)^{\frac{\be}{d}}  \notag  \\
& \leq
 C m^{2-\frac{\be}{d}} \leq   C m^{2-\frac{s}{2d}}.
 \label{edge_sum}
\end{align}
Here, $\be\le s<d$ by \eqref{deg_rel},
so that the factor $K_d^{\be/d}$, bounded by $\max\{1, K_d\}$,
 can be absorbed into
the constant $C$. The last inequality in \eqref{edge_sum}
is due to   $\be\ge s/2$.

Combining \eqref{conv_bd} and \eqref{edge_sum}, and interchanging
$y,z$ when $\al>\be$, yields
\begin{align}\label{fixed_diagram_sum}
 &\sum_{x,y,z\in A}'
 |r(x,y)|^\al
 |r(x,z)|^\be
 |r(y,z)|^c
 \leq
 C\eps^{\min\{\al,\be\}+c-d}
 m^{2-\frac{s}{2d}}.
\end{align}

\smallskip

We finally sum over the Hermite orders.
In view of \eqref{wick3} and \eqref{fixed_diagram_sum},
each fixed pair
$(\al,\be)$ with $\al+\be=s$ contributes to the left side of
\eqref{diag_mix},
and its contribution is bounded by

\noi
\begin{align*}
 &C|b|m^{2-\frac{s}{2d}}
 \sum_{c\ge d-\min(\al,\be)}
 |u_{\al+c}v_{\be+c}|
 (1+c)^{\frac{s}{2}}\eps^{\min\{\al,\be\}+c-d}.
\end{align*}

\noi
Since $\eps\in(0,\frac{1}{3})$, we have

\noi
\begin{align*}
\sup_{\substack{\al,\be\ge0\\ \al+\be=s}} \sup_{c\ge d- \min\{\al,\be\}}
 (1+c)^{\frac{s}{2}}\eps^{\min\{\al,\be\}+c-d}
 <\infty.
\end{align*}
Therefore, we can apply the above uniform bound
and  Cauchy--Schwarz to get
\begin{align*}
 \sum_{c\ge d-\min\{\al,\be\}}
 |u_{\al+c}v_{\be+c}|
 (1+c)^{\frac{s}{2}}\eps^{\min\{\al,\be\}+c-d}
&\leq
 C_{d,s,\eps}
 \bigg(\sum_{q\ge d}u_q^2\bigg)^{\frac{1}{2}}
 \bigg(\sum_{\ell\ge d}v_\ell^2\bigg)^{\frac{1}{2}} \\
&=
 C_{d,s,\eps}\|f_1\|_2\|f_2\|_2.
\end{align*}
There are only finitely many pairs $(\al,\be)$ with
$\al+\be=s$.
As a result, we get

\noi
\begin{align}\label{diag_mix_finite}
 \sum_{x,y,z\in A}'
 |\E[B(Z_x)f_1(Z_y)f_2(Z_z)]|
 \le
 C|b|m^{2- \frac{s}{2d}}
 \|f_1\|_2\|f_2\|_2,
\end{align}
when $f_1, f_2$ have finite Hermite expansions, or equivalently,
when they are polynomials.
Note that the constant $C$ in \eqref{diag_mix_finite}
does not depend on the degrees of the polynomials $f_1, f_2$.

It remains to remove the finite-expansion assumption.  Let
$f_{1,N} = \sum_{q=d}^N u_q h_q$
and $f_{2,N} = \sum_{\ell=d}^N v_\ell h_\ell$.
 Then,
 $f_{i,N} \to f_i$
in  $L^2(\ga)$
and
$ \|f_{i,N}\|_2\le\|f_i\|_2$ for $i=1,2$.

For distinct $x,y,z$, let $\Sigma$ be the covariance matrix of
$(Z_x,Z_y,Z_z)$.   By \eqref{small_c1} and Gershgorin's circle theorem
(see, e.g., \cite[Theorem~1.1]{Var04}),
the spectrum of the covariance matrix $\Sigma$
is contained in $ [1-2\eps,1+2\eps]\subset(0,2)$.
In particular, $\Sigma$ is positive definite, so the law of
$(Z_x,Z_y,Z_z)$ is nondegenerate and admits a density $L$ with
respect to $\ga^{\otimes3}$.  Moreover, the operator norm $\|\Sigma\|<2$ implies that
$2\Sigma^{-1}-I$ is positive definite, and hence
$L\in L^2(\ga^{\otimes3})$.
Consequently,

\noi
\begin{align*}
\Big\vert  \E\big[
& B(Z_x)(f_{1,N}-f_1)(Z_y)f_{2,N}(Z_z)  \big] \Big\vert  \\
 &= \left\vert  \int_{\R^3 } B(w_1)(f_{1,N}-f_1)(w_2)f_{2,N}(w_3) L(w_1,
 w_2, w_3) \ga^{\otimes 3}(dw_1, dw_2, dw_3) \right\vert\\
 & \le
  \| L\|_2  |b|
 \|f_{1,N}-f_1\|_2
 \|f_{2,N}\|_2
 \xrightarrow{N\to+\infty}0,
\end{align*}
and the analogous estimate holds for the second factor.  Thus,
for every distinct triple $(x,y,z)$,
\begin{align*}
 \E[B(Z_x)f_{1,N}(Z_y)f_{2,N}(Z_z)]
 \longrightarrow
 \E[B(Z_x)f_1(Z_y)f_2(Z_z)].
\end{align*}
Since the primed sum over $A$ is finite, we may pass to the limit
in \eqref{diag_mix_finite}.  This proves \eqref{diag_mix} and
hence  completes the proof of Lemma  \ref{diag_lem}.
\qedhere

\end{proof}

\begin{remark}\rm
Parts~\textup{(i)} and \textup{(ii)} of Lemma~\ref{diag_lem}
are direct consequences of
\cite[Lemma~1]{BS13a}, since all factors have Hermite rank at least
$d$.  Part~\textup{(iii)} requires a refinement.
Indeed, applying \cite[Lemma~1]{BS13a} with $p=3$ and only
$f_1,f_2$ subject to the Hermite-rank-$d$ condition
(that is, with $\alpha=2$ in the notation of \cite{BS13a})
would yield only an $O(m^2)$ bound.  Here we exploit the
additional fact that the remaining factor $B=bh_s$ belongs to the
single chaos of order $s<d$.  In the three-row diagram expansion the
edge multiplicities satisfy $\al+\be=s$, so one of $\al,\be$ is at
least $s/2$.  Retaining the corresponding correlation power and
using
the $d$th-power covariance row bound yields the additional factor
$m^{-\frac{s}{2d}}$, and hence the sharper bound
$m^{2-\frac{s}{2d}}$.
\end{remark}

\subsection{A maximal estimate from adjacent-block bounds}
\label{max_sec}

The following lemma converts a strict power gain for two adjacent
blocks into a maximal second-moment estimate.  It is a consequence
of the classical two-increment maximal principle; see, e.g.,
\cite[Section~7.3, Lemma~7.15]{Sag20}.  We include a direct dyadic
proof adapted to the block estimates below.  We apply the lemma to
the rank-corrected truncation in the proof of
Theorem~\ref{ns_max}.  The truncation and low-Hermite-correction
arguments are inspired by \cite{BHB25}; our objective is an estimate
uniform over nonstationary covariance matrices satisfying
\eqref{ns_rows}.

\begin{lemma}[\textsf{Adjacent-block maximal lemma}]\label{max_lem}
Let  $X_1,\ldots,X_N$ be centered random variables.  For an integer
interval $I\subset\{1,\ldots,N\}$, write $S_I=\sum_{i\in I}X_i$.
Suppose that
\begin{align}
 \E[S_I^2]&\le C_0|I|,\label{abs_var}\\
 \E[S_I^2S_J^2]&\le C_1N^{1-\al}
 (|I|+|J|)^{1+\al}\label{abs_adj}
\end{align}
for all $I$ and all adjacent disjoint $I,J$ {\rm(}as illustrated in Figure \ref{fig_adj}{\rm)},
 where $\al>0$.
Then
\begin{align}\label{abs_max}
 \E\left[
  \max_{0\le m\le N}
  \left|\sum_{k=1}^mX_k\right|^2
 \right]\le C_{\al,C_0,C_1}N.
\end{align}
\end{lemma}

\begin{figure}[t]
\centering
\begin{tikzpicture}[x=.40cm,y=.40cm,>=Latex]

\definecolor{Iblue}{RGB}{10,78,150}
\definecolor{Jred}{RGB}{210,15,15}

\draw[line width=.8pt]
  (1.5,10.2) -- (24.8,10.2);

\foreach \x in {
  2.1,3.7,5.3,6.9,8.5,
  10.1,11.7,13.3,14.9,
  16.5,18.1,19.7,21.3,22.9,24.2
}
{
  \fill (\x,10.2) circle[radius=.11];
}

\fill[Iblue,opacity=.15]
  (9.3,9.65) rectangle (14.4,10.75);

\fill[Jred,opacity=.15]
  (14.55,9.65) rectangle (19.2,10.75);

\draw[
  Iblue,
  line width=.9pt,
  decorate,
  decoration={brace,amplitude=5pt}
]
(9.3,11.25) -- (14.4,11.25);

\node[Iblue] at (11.85,12.45)
{$I$};

\draw[
  Jred,
  line width=.9pt,
  decorate,
  decoration={brace,amplitude=5pt}
]
(14.55,11.25) -- (19.2,11.25);

\node[Jred] at (16.9,12.45)
{$J$};

\node[anchor=north] at (2.1,9.6) {$1$};
\node[anchor=north] at (3.7,9.6) {$2$};
\node[anchor=north] at (5.3,9.6) {$3$};
\node[anchor=north] at (6.9,9.6) {$4$};
\node[anchor=north] at (8.5,9.6) {$\cdots$};

\node[anchor=north] at (20.8,9.6) {$\cdots$};
\node[anchor=north] at (24.2,9.6) {$N$};

\draw[
  -{Latex[length=2.2mm,width=1.5mm]},
  line width=.8pt
]
(14.4,8.15) -- (14.4,9.45);

\node[anchor=north]
  at (14.4,7.95)
{\small adjacent};

\end{tikzpicture}

\caption{\small The integer intervals $I$ and $J$ are disjoint  with no
gap between them.}
\label{fig_adj}
\end{figure}
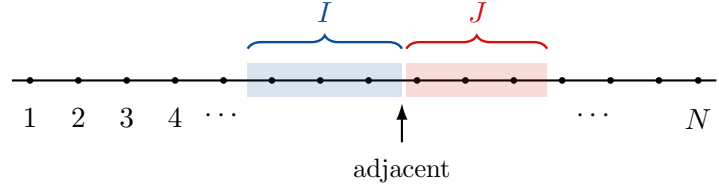

\begin{remark}\rm
\label{rem_mor}
It is useful to compare  Lemma \ref{max_lem}
with the classical
moment-to-maximum inequality of M\'oricz
\cite[Theorem~1]{Mor76}.  In a simplified form, M\'oricz's result
starts from a one-block moment estimate
$
 \E|S_I|^p\le C|I|^\beta
 $
 with 
 $\beta>1$,
and yields a maximal estimate of the same power order.
In the present setting, the natural one-block estimate is only
$
 \E[S_I^2]\le C|I|.
$
Thus the one-block second moment lies at the borderline $\beta=1$,
and the elementary dyadic argument based on this estimate gives
$\E\big[ \max_{m\le N}|S_{(0,m]}|^2 \big] \leq C N(\log(2N))^2.$
Uniform finite-variance approximation therefore requires a
sharper maximal bound.

The additional information available here has a different form.
For adjacent disjoint blocks $I,J$, we prove the bound \eqref{abs_adj},
in which the strict power gain is  sufficient
to recover the diffusive maximal bound \eqref{abs_max}
through a dyadic argument.
In this sense,  Lemma \ref{max_lem} is complementary to M\'oricz's principle:
M\'oricz exploits a superlinear moment bound for a single block,
whereas here the single-block second moment remains linear and the
strict gain appears only after coupling two adjacent blocks.
\end{remark}

\begin{proof}[Proof of Lemma \ref{max_lem}]
The proof is divided into three steps.  Throughout, $(\ell,r]$
denotes the integer interval $(\ell,r]\cap\mathbb N$.

\smallskip
\noi
{\bf Step 1: Reduction to a two-sided cut maximal function.}
Choose the smallest dyadic integer $M=2^L$ such that $M\ge N$,
so that $N\le M<2N$, and set
$X_k=0$ for $N<k\le M$.  This allows us to work on a complete
dyadic tree down to singleton intervals, while changing the relevant
scales only by an absolute constant.
For an integer interval $D=(\ell,r]\subset(0,M]$, define
\begin{align}\label{cut_def}
 R(D)
 :=
 \max_{\ell\le m\le r}
 \min\big\{
 |S_{(\ell,m]}|,
 |S_{(m,r]}|
 \big\}.
\end{align}
Thus $R(D)$ measures, over all cuts of $D$, the largest possible
value of the smaller of the two partial sums created by the cut.

Define  $P_m=\sum_{k=1}^mX_k$.
Then, we have
$S_{(0,m]}=P_m$ and
$S_{(m,M]}=P_N-P_m.$
Therefore,
$\min\{|P_m|,|P_N-P_m|\}\le R((0,M])$.
Moreover,
\begin{align}\label{prefix_bd}
 |P_m|
 \le
 |P_N|
 +
 \min\{|P_m|,|P_N-P_m|\}
 \le
 |P_N|+R((0,M]).
\end{align}

\noi
Hence,  due to Assumption \eqref{abs_var}, it remains to prove
$
 \|R((0,M])\|_2 \leq C\sqrt N.
$

\smallskip
\noi
{\bf Step 2: Dyadic recursion for the cut maximal function.}
Let $D=(\ell,r]$ be a dyadic interval,
 $a=\frac{\ell+r}{2}$,
and write $D^-=(\ell,a]$ and $D^+=(a,r]$.
Set $z_D= \min\{|S_{D^-}|,|S_{D^+}|\}.$
We claim that
\begin{align}\label{dy_rec}
 R(D)
 \le
 z_D+\max\{R(D^-),R(D^+)\}.
\end{align}

To prove this, first note the elementary inequality\footnote{To
verify \eqref{min_ineq}, if $|a|\le |b|$, then it is immediate.
  If $|b|<|a|$, then
$
 \min\{|a|,|b+c|\}-|b|
 =\min\{|a|-|b|,\ |b+c|-|b|\}
 \le\min\{|a+b|,\ |c|\},
$
which implies \eqref{min_ineq}.}
\begin{align}\label{min_ineq}
 \min\{|a|,|b+c|\}
 \le
 \min\{|a|,|b|\}
 +
 \min\{|a+b|,|c|\}.
\end{align}

Consider  $m\in D^-$.  Then,
we apply   \eqref{min_ineq}
with $a_1:=S_{(\ell,m]}$, $b_1:=S_{(m,a]}$, and
$c_1:=S_{D^+}$
to get

\noi
\begin{align*}
 \min\big\{  |S_{(\ell,m]}|,   |S_{(m,r]}|  \big\}
 =  \min\big\{  |a_1 |,   |b_1 + c_1 |  \big\}
 &\leq
 \min\{|a_1|,|b_1|\} +  \min\{|a_1+b_1|,|c_1|\} \\
 & \leq R(D^-)+z_D.
\end{align*}

\noi
For $m\in D^+$, the same argument  yields
$\min\big\{ |S_{(\ell,m]}|,   |S_{(m,r]}| \big\} \leq  z_D+R(D^+).$
Taking the maximum over all  $m\in D$ proves \eqref{dy_rec}.

At dyadic level $j$, the interval $(0,M]$ is partitioned into
$2^j$ dyadic subintervals
\begin{align*}
 D_{j,k}
 =
 \big(kM2^{-j},(k+1)M2^{-j}\big],
 \quad k=0,\ldots,2^j-1,
\end{align*}
each of length
$ \ell_j=M2^{-j}$,
  $0\le j<L$.
Define
\begin{align}\label{def_ZJ}
 Z_j :=  \max\{  z_D:  D\text{ is a dyadic interval   at level $j$} \}.
\end{align}
Iterating \eqref{dy_rec} from the root down to singleton sets gives
\begin{align}\label{cut_sum}
 R((0,M])
 \leq
 \sum_{j=0}^{L-1}Z_j.
\end{align}
Indeed, at each generation one pays at most the largest $z_D$ on that
level, while at the final level the cut maximal function of a singleton
is zero.

\smallskip
\noi
{\bf Step 3: Estimate each dyadic scale and sum over the scales.}
Fix a level $j$.  There are
$M/\ell_j$
dyadic intervals of length $\ell_j$.  For each such interval  $D$,

\noi
\begin{align}\label{ZD4}
 z_D^4
 =
 \min\{|S_{D^-}|,|S_{D^+}|\}^4
 \le
 S_{D^-}^2S_{D^+}^2.
\end{align}
After intersecting $D^-$ and $D^+$ with $\{1,\ldots,N\}$, the two
sets remain adjacent integer intervals.  If one of them is empty, then
$z_D=0$.  Otherwise, by \eqref{abs_adj} and \eqref{ZD4},
we have

\noi
\begin{align*}
 \E[z_D^4]
 \le
 C_1N^{1-\al}\ell_j^{1+\al}.
\end{align*}
By the definition \eqref{def_ZJ} of $Z_j$,
we can write, with $M<2N$,

\noi
\begin{align}
 \E[Z_j^4]
 \le
 \sum_{|D|=\ell_j}\E[z_D^4]
 \le
 \frac{M}{\ell_j}
 C_1N^{1-\al}\ell_j^{1+\al}
 =
 C_1MN^{1-\al}\ell_j^\al 
 \le
 2C_1N^{2-\al}\ell_j^\al .
 \label{Z4_bd}
\end{align}

\noi
Therefore,
 $\|Z_j\|_2 \le  \|Z_j\|_4  \le C_1'  N^{\frac{1}{2}- \frac{\al}{4}}
 \ell_j^{\al/4}.$
Combining \eqref{cut_sum} with Minkowski's inequality yields,
with $M < 2N$,
\begin{align*}
 \|R((0,M])\|_2
 \le
 \sum_{j=0}^{L-1}\|Z_j\|_2
 &\le
 C_1' N^{\frac{1}{2}- \frac{\al}{4}}
 \sum_{j=0}^{L-1}
 (M2^{-j})^{\al/4}\\
 &\leq
2^{\frac{\al}{4}} C_1'  N^{\frac{1}{2} }
 \sum_{j=0}^{\infty}
 2^{-j\al/4}.
\end{align*}

\noi
Since $\al>0$, we get
$ \|R((0,M])\|_2 \le C_{\al,C_1}\sqrt N$.\footnote{This is precisely where
the
 condition $\al>0$ is used; if
$\al=0$, the sum over the dyadic scales would be of order
$\log N$ rather than uniformly bounded.}
Finally,   from \eqref{abs_var}  and  \eqref{prefix_bd},
we get
\[
 \left\|
 \max_{1\le m\le N}|P_m|
 \right\|_2
 \le
 \|P_N\|_2+\|R((0,M])\|_2
 \le
 C_{\al,C_0,C_1}\sqrt N.
\]
Hence, the proof is completed.
\qedhere

\end{proof}

\section{The maximal inequality above the endpoint}
\label{lp_sec}

We prove Theorem~\ref{ns_max} for every real $p>2$ by a
straightforward modification of the argument in \cite{NN20},
replacing the stationary covariance sum by the row-sum bound
\eqref{ns_rows}.  We keep the proof brief and we emphasize only the
dyadic summation step, which is summable for $p>2$ but becomes
critical at $p=2$.

\subsection{Malliavin--Meyer block estimates}
\label{lp_block_sec}

We record the block estimate in the form needed below and refer to
\cite{NN20} for its derivation and further details.

Represent $Z_i=\mathcal{X} (e_i)$ in an isonormal Gaussian family over
a finite-dimensional real Hilbert space $\mathfrak H$, with
$\|e_i\|_{\mathfrak H}=1$ and
$\langle e_i,e_j\rangle_{\mathfrak H}=r_{ij}:=\E[Z_iZ_j]$.
This representation also allows singular covariance matrices;
see \cite[Proposition~7.2.3]{NP12}.
Let $D^k$ denote the $k$th Malliavin derivative and
$\delta^d$ the adjoint of $D^d$.
For $p>2$, Meyer's inequality \cite{Mey84}, in the divergence
form used in \cite[Section~2.1]{NN20}, gives
\begin{align}\label{meyer_div}
 \|\delta^d(U)\|_p
 \le C_{p,d}\sum_{k=0}^d
 \|D^kU\|_{L^p(\Omega;\mathfrak H^{\otimes(d+k)})}.
\end{align}
The constant is independent of the dimension of $\mathfrak H$.

For a polynomial $f=\sum_{q=d}^Q c_qH_q$, define its
Hermite shift by
\begin{align}\label{shift_def}
 f_d=\sum_{q=d}^Q c_qH_{q-d}.
\end{align}
The representation and Sobolev estimates in
\cite[Lemmas~2.1--2.2 and their proofs]{NN20} yield
\begin{align}
 f(Z_i)
 &=\delta^d\big(f_d(Z_i)e_i^{\otimes d}\big),
 \label{shift_identity}\\
 \sum_{k=0}^d\|f_d^{(k)}\|_p
 &\le C_{p,d}\|f\|_p,
 \label{shift_sob}
\end{align}
where $f_d^{(k)}$ is the ordinary $k$th derivative.
Crucially, the constant in \eqref{shift_sob} is independent
of the degree $Q$.
The only modification in the block argument below is to use
the row-sum bound \eqref{ns_rows} in place of the stationary
covariance sum.

\begin{proposition}[\textsf{Uniform $L^p$ block bound}]
\label{lp_block}
Under \eqref{ns_rows}, for $p>2$, every $f\in L^p(\ga)$ of
Hermite rank at least $d$, and every index set
$I\subset\{1,\ldots,n\}$,
we have 

\noi
\begin{align}\label{lp_block_bd}
 \left\|\sum_{i\in I}f(Z_i)\right\|_p
 \le C_{p,d}\sqrt{K|I|}\,\|f\|_p.
\end{align}
\end{proposition}

\begin{proof}
First let $f$ be a polynomial of rank at least $d$ and set $U_I=\sum_{i\in
I}f_d(Z_i)e_i^{\otimes d}$.
By \eqref{shift_identity}, the sum in \eqref{lp_block_bd} equals
$\delta^d(U_I)$, and
\begin{align*}
 D^kU_I=\sum_{i\in I}f_d^{(k)}(Z_i)e_i^{\otimes(d+k)},
 \quad 0\le k\le d.
\end{align*}
Since $p/2\ge1$, Minkowski's and H\"older's inequalities yield
\begin{align*}
 \|D^kU_I\|_{L^p(\Omega;\mathfrak H^{\otimes(d+k)})}^2
 &\le\sum_{i,j\in I}|r_{ij}|^{d+k}
       \|f_d^{(k)}(Z_i)f_d^{(k)}(Z_j)\|_{\frac{p}{2}}\\
 &\le\|f_d^{(k)}\|_p^2
       \sum_{i,j\in I}|r_{ij}|^{d+k}
 \le K  |I|  \,\|f_d^{(k)}\|_p^2.
\end{align*}
Here all marginals are standard normal and $|r_{ij}|\le1$.
Combining this with \eqref{meyer_div} and \eqref{shift_sob} proves
\eqref{lp_block_bd} for polynomials, uniformly in their degree.

To pass to a general $f\in L^p(\ga)$, choose polynomials
$a_m\to f$ in $L^p(\ga)$; see, e.g.,
\cite[Proposition~1.1.5]{NP12}.
It is well known that the finite-rank projection $\Pi_{<d}$ is bounded on $L^p(\ga)$.
Indeed, as a consequence of H\"older's inequality, 
we have, with $p' = p/(p-1)$,

\noi
\begin{align*}
 \|\Pi_{<d}a\|_p
 &\le \sum_{q=0}^{d-1}
 |\langle a,h_q\rangle_{L^2(\ga)}|\,\|h_q\|_p
 \le \bigg( \sum_{q=0}^{d-1}  \|h_q\|_p\|h_q\|_{p'} \bigg) \|a\|_p.
\end{align*}

Since $\Pi_{<d}f=0$, we deduce from the above bound that 
$\Pi_{<d}a_m = \Pi_{<d}(a_m-f) \to 0$ in $L^p(\ga)$,
and thus
the polynomials
$f_m=a_m-\Pi_{<d}a_m$
have Hermite rank at least $d$ and satisfy
$f_m\to f$ in $L^p(\ga)$.  For every fixed finite set $I$,
$\|\sum_{i\in I}(f_m-f)(Z_i) \|_p
 \le |I|\,\|f_m-f\|_p
 \to 0.$
Passing to the limit yields the desired bound.
\qedhere

\end{proof}

\subsection{Dyadic summation for every real \texorpdfstring{$p>2$}{p>2}}
\label{lp_dyadic}

We first record the combinatorial decomposition used below
and again in Section~\ref{max_pf}.

\begin{lemma}[\textsf{Dyadic decomposition of initial segments}]
\label{com_lem}
Let $N=2^L$ for an integer $L\ge0$.
For each dyadic length $\ell\in\{1,2,4,\ldots,N\}$, define
\begin{align}\label{com_part}
 \mathcal D_\ell
 =
 \left\{
 \{j\ell+1,\ldots,(j+1)\ell\}:
 j=0,\ldots,\frac{N}{\ell}-1
 \right\}.
\end{align}
Thus $\mathcal D_\ell$ partitions $\{1,\ldots,N\}$ into
$N/\ell$ intervals of length $\ell$.

Every initial segment $\{1,\ldots,m\}$, $0\le m\le N$,
is a disjoint union of canonical dyadic intervals, with
at most one interval from each $\mathcal D_\ell$.
Consequently, for arbitrary real numbers $x_1,\ldots,x_N$,
\begin{align}\label{use2a}
 \max_{0\le m\le N}
 \left|\sum_{k=1}^m x_k\right|
 \le
 \sum_{\substack{\ell\le N\\ \ell\ {\rm dyadic}}}
 \max_{D\in\mathcal D_\ell}
 \left|\sum_{k\in D}x_k\right|.
\end{align}
\end{lemma}

\begin{proof}
The case $m=0$ corresponds to the empty union.
For $1\le m\le N$, write its binary expansion as
$m=\sum_{\nu=1}^r2^{a_\nu}$
with $L\ge a_1>a_2>\cdots>a_r\ge0.$
Set $s_0=0$ and
$s_\nu=\sum_{\mu=1}^\nu2^{a_\mu}$ for $1\le\nu\le r$.
Since $s_r=m$, 
the intervals
 $D_\nu=\{s_{\nu-1}+1,\ldots,s_\nu\}$,
$1\le\nu\le r$,
are pairwise disjoint
 and satisfy
$
 \{1,\ldots,m\}=\bigcup_{\nu=1}^rD_\nu.
$
Moreover, $|D_\nu|=2^{a_\nu}$ and $s_{\nu-1}$ is a multiple
of $2^{a_\nu}$, because $a_\mu>a_\nu$ for every $\mu<\nu$.
Hence $D_\nu\in\mathcal D_{2^{a_\nu}}$.
The exponents are distinct, so at most one interval of each
dyadic length occurs.
See Figure~\ref{fig2} for the example $m=23$.

For each $m$, the triangle inequality therefore gives
\begin{align*}
 \left|\sum_{k=1}^m x_k\right|
 &\le \sum_{\nu=1}^r
       \left|\sum_{k\in D_\nu}x_k\right|
       \le
 \sum_{\substack{\ell\le N\\ \ell\ {\rm dyadic}}}
 \max_{D\in\mathcal D_\ell}
 \left|\sum_{k\in D}x_k\right|.
\end{align*}
Taking the maximum over $m$ proves \eqref{use2a}.
\end{proof}

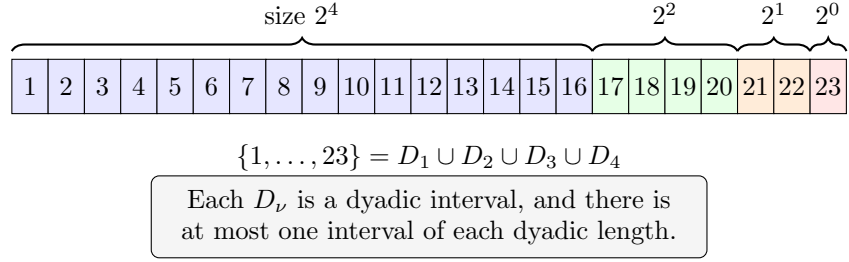
\begin{figure}[t]
\centering
\begin{tikzpicture}[
    x=0.48cm,
    y=0.72cm,
    every node/.style={font=\small}
]

\foreach \k in {1,...,16}{
    \draw[fill=blue!10] (\k-1,0) rectangle (\k,1);
    \node at (\k-0.5,0.5) {\k};
}

\foreach \k in {17,...,20}{
    \draw[fill=green!10] (\k-1,0) rectangle (\k,1);
    \node at (\k-0.5,0.5) {\k};
}

\foreach \k in {21,...,22}{
    \draw[fill=orange!15] (\k-1,0) rectangle (\k,1);
    \node at (\k-0.5,0.5) {\k};
}

\draw[fill=red!10] (22,0) rectangle (23,1);
\node at (22.5,0.5) {23};

\draw[
    decorate,
    decoration={brace,amplitude=5pt},
    thick
]
(0,1.15) -- (16,1.15)
node[midway,yshift=0.5cm] {\small size $2^4$};

\draw[
    decorate,
    decoration={brace,amplitude=5pt},
    thick
]
(16,1.15) -- (20,1.15)
node[midway,yshift=0.5cm] {\small $2^2$};

\draw[
    decorate,
    decoration={brace,amplitude=5pt},
    thick
]
(20,1.15) -- (22,1.15)
node[midway,yshift=0.5cm] {\small $2^1$};

\draw[
    decorate,
    decoration={brace,amplitude=5pt},
    thick
]
(22,1.15) -- (23,1.15)
node[midway,yshift=0.5cm] {\small $2^0$};

\node at (11.5,-0.8)
{\small $\{1,\ldots,23\}=D_1\cup D_2\cup D_3\cup D_4$};

\node[
    draw,
    rounded corners=3pt,
    fill=gray!8,
    text width=7cm,
    align=center,
    inner sep=5pt
] at (11.5,-1.9)
{\small Each $D_\nu$ is a dyadic interval, and there is at most one
interval of each dyadic length.};

\end{tikzpicture}
\caption{\small Dyadic decomposition of $\{1,\ldots,m\}$
with $m=23=2^4+2^2+2^1+2^0$.}
\label{fig2}
\end{figure}

Now we are ready to present the proof of  Theorem~\ref{ns_max} 
in the high moment regime.

\begin{proof}[\textbf{Proof of Theorem~\ref{ns_max} for $p>2$}]
Let $N=2^L$ be the smallest power of two with $N\ge n$, so that
$n\le N<2n$, and define
$f(Z_k)=0$ for $n<k\le N$.
For $D\subset\{1,\ldots,N\}$, write
\[
 S_D=\sum_{k\in D}f(Z_k).
\]
By Lemma~\ref{com_lem}, every initial segment
$\{1,\ldots,m\}$ is the disjoint union of at most one interval
from each dyadic partition.  Hence
\begin{align*}
 \max_{0\le m\le n}
 \left|\sum_{k=1}^m f(Z_k)\right|
 \le
 \sum_{j=0}^L
 \max_{D\in\mathcal D_{N2^{-j}}}|S_D|.
\end{align*}

For each $0\le j\le L$, the partition
$\mathcal D_{N2^{-j}}$ contains $2^j$ intervals, each of
length $N2^{-j}$. 
Then,  Proposition~\ref{lp_block}, applied to
$D\cap\{1,\ldots,n\}$, gives

\noi
\begin{align*}
 \left\|  \max_{D\in\mathcal D_{N2^{-j}}}|S_D|  \right\|_p
 \le
 \left(  \sum_{D\in\mathcal D_{N2^{-j}}}   \|S_D\|_p^p  \right)^{1/p}
 &\le
 \left[  2^j  \left( C_{p,d}\sqrt{K N2^{-j}}\,\|f\|_p  \right)^p \right]^{1/p}\\
 &= C_{p,d}\sqrt{KN}\,\|f\|_p\, 2^{-j(\frac12-\frac1p)}.
\end{align*}

\noi
For $p>2$, the geometric series satisfies
\begin{align}\label{lp_geometric}
 \sum_{j=0}^L
 2^{-j(\frac12-\frac1p)}
 &\le
 \sum_{j=0}^\infty
 2^{-j(\frac12-\frac1p)}
 =
  \big( 1-2^{\frac1p-\frac12}  \big)^{-1}.
\end{align}

\noi
Therefore, by Minkowski's inequality and $N<2n$,
\begin{align}\label{lp_max_proof}
 \left\|
 \max_{0\le m\le n}
 \left|\sum_{k=1}^m f(Z_k)\right|
 \right\|_p
 &\le
 C_{p,d}\sqrt{KN}\,\|f\|_p
 \sum_{j=0}^L
 2^{-j(\frac12-\frac1p)}\le
 \frac{C_{p,d}\sqrt{Kn}}
 {1-2^{\frac1p-\frac12}}\,
 \|f\|_p.
\end{align}
This proves \eqref{ns_bdd} for $p>2$.
\qedhere

\end{proof}

\section{The critical \texorpdfstring{$L^2$}{L2} maximal inequality}
\label{ns_sec}

We follow Steps~\eqref{step_trunc}--\eqref{step_partition}
from the Introduction.  Rank-preserving truncation and the
low--high chaos analysis give the adjacent-block bound needed
for Lemma~\ref{max_lem}.  We then restore the discarded terms
in $L^2$ and remove the separation assumption by a finite
partition.

\subsection{Rank-preserving truncation}
\label{clip}

We carry out Step~\eqref{step_trunc}.  We first truncate $\varphi$ and then
subtract the low Hermite projections created by this operation.
The resulting function has rank at least $d$, but need not be
bounded.

More precisely, throughout Sections~\ref{clip}--\ref{adj} let
$\varphi\in L^2(\ga)$ have vanishing Hermite coefficients below
order $d$.  Define
\begin{align}\label{trunc_def}
 T_n=\varphi\ind_{\{|\varphi|\le\sqrt n\}},
 \quad \be_{q,n}=\langle T_n,h_q\rangle_{L^2(\ga)},
 \quad 0\le q<d,
\end{align}
with $h_q$ as in \eqref{def_hq},
and let $\Pi_{<d}$ be the orthogonal projection onto
$\operatorname{span}\{h_0,\ldots,h_{d-1}\}$.  Set
\begin{align}\label{g_def}
 g_n=T_n -\Pi_{<d}T_n
 =T_n-\sum_{q=0}^{d-1}\be_{q,n}h_q.
\end{align}

\begin{proposition}[\textsf{Rank-preserving truncation}]\label{clip_prop}
The functions $g_n$ have vanishing Hermite coefficients below
order $d$ and satisfy
\begin{align}\label{g_mom}
 \|g_n\|_2\le\|\varphi\|_2,
 \quad\|g_n-\varphi\|_2\longrightarrow0,
 \quad\E[g_n(\cZ)^4]\le Cn.
\end{align}
Moreover,
\begin{align}\label{clip_coef}
 |\be_{0,n}|\le n^{-\frac{1}{2}}\|\varphi\|_2^2,
 \quad
 |\be_{q,n}|\le C_qn^{-1/2}(1+\log n)^{q/2},
\quad 1\le q<d.
\end{align}
The constants are uniform over $\|\varphi\|_2\le M$ for each
fixed $M<\infty$.
\end{proposition}

 \begin{proof}
Put  $E_n=\{|\varphi|>\sqrt n\}$ and
$ q_n=\E[\varphi(\cZ)^2\ind_{E_n}(\cZ)].$
Then, $q_n\to0$ by dominated convergence, and
\begin{align}\label{En_prob}
 \ga(E_n)\le \frac{q_n}{n}\le\frac{\|\varphi\|_2^2}{n}.
\end{align}
Since the Hermite coefficients of $\varphi$ below order $d$ vanish,
\begin{align}\label{beta_tail}
 \be_{q,n}
 =-\E[\varphi(\cZ)h_q(\cZ)\ind_{E_n}(\cZ)],
 \quad 0\le q<d.
\end{align}

For $q=0$, on $E_n$ we have
$|\varphi|\le n^{-\frac{1}{2}}\varphi^2$, and thus
\begin{align}\label{beta0}
 |\be_{0,n}|
 \le q_n n^{-\frac{1}{2}}
 \le n^{-\frac{1}{2}}\|\varphi\|_2^2.
\end{align}
For $1\le q<d$, Cauchy--Schwarz gives
\begin{align}\label{beta_cs}
 |\be_{q,n}|
 \le q_n^{\frac{1}{2}}
 \left(\int_{E_n}h_q^2\,d\ga\right)^{\frac{1}{2}}.
\end{align}
Fix $M<\infty$ with $\|\varphi\|_2\le M$.  For sufficiently
large $n$, we have $\frac{M^2}{n}\le\min\{\frac{1}{2},e^{1-q}\}$,
while 
the function $p\mapsto p(1+\log(\frac{1}{p}))^q$ is increasing on 
$(0, e^{1-q} ]$. 
Applying Lemma~\ref{rearr_lem}\textup{(i)} and
\eqref{En_prob} therefore yields
\begin{align*}
 \int_{E_n}h_q^2\,d\ga
 \le C_q n^{-1}(1+\log n)^q.
\end{align*}
Since $q_n^{1/2}\le\|\varphi\|_2$, \eqref{beta_cs} implies
\begin{align}\label{beta_bd}
 |\be_{q,n}|
 \le C_q n^{- 1/2}(1+\log n)^{q/2},
 \quad 1\le q<d.
\end{align}
The finitely many remaining values of $n$ are absorbed into the
constant.  When $\ga(E_n)=0$, \eqref{beta_tail} gives
$\be_{q,n}=0$.  The constants above can be chosen uniformly when
$\|\varphi\|_2$ ranges over a fixed bounded set.
This proves \eqref{clip_coef}.

Recall from \eqref{g_def} that
$g_n= T_n -\Pi_{<d}T_n$.  Since
$\Pi_{<d}\varphi=0$, the contraction property of the
orthogonal projection gives
\begin{align}\label{g_l2err}
 \begin{aligned}
 \|g_n\|_2
 &\le \|T_n\|_2
 \le \|\varphi\|_2,\\
 \|g_n-\varphi\|_2^2
 &\le \|T_n-\varphi\|_2^2
 =q_n\longrightarrow0.
 \end{aligned}
\end{align}

Finally, let
\begin{align}\label{def_PN}
 P_n=\Pi_{<d}T_n=\sum_{q=0}^{d-1}\be_{q,n}h_q.
\end{align}
Since $|T_n|\le\sqrt n$,
$
 \E[T_n(\cZ)^4]
 \le n\E[T_n(\cZ)^2]
 \le n\|\varphi\|_2^2.
$
Moreover, from
$
 |\be_{q,n}|
 \le \|T_n\|_2
 \le \|\varphi\|_2,
$
we can easily deduce that
\begin{align}\label{bdd_PN}
 \sup_{n\ge1}\E[P_n(\cZ)^4]<\infty,
 \end{align}
and consequently,
 $\E[g_n(\cZ)^4]\le Cn.$
This proves \eqref{g_mom} and thus concludes the  proof.
\qedhere

\end{proof}

Recall that Step \eqref{step_diagrams},  the separated diagram estimates, was established in Lemma \ref{diag_lem}. We now turn to Step \eqref{step_square}.

\subsection{The square of the rank-corrected truncation}
\label{square}

In Step~\eqref{step_square}, the adjacent-block estimate requires
products of two squared block sums.  We therefore single out the
low Hermite chaoses of $g_n^2$.
Let  $\cZ\sim\cN(0,1)$ and
\begin{align}\label{sq_dec}
  g_n^2- \E[g_n(\cZ)^2] =\sum_{q=1}^{d-1}B_{q,n}+\mathcal H_n,
\end{align}
where $B_{q,n} =  \E[ g_n^2(\cZ) h_q(\cZ)] h_q$ is the $q$th Hermite-chaos
projection for $q\ge1$, and
$\mathcal H_n = \sum_{q\geq d} B_{q,n}$
has Hermite rank at least $d$.
 When $d=1$, the sum in
\eqref{sq_dec} is empty.

\begin{lemma}
\label{sq_lem}
Fix $M<\infty$ and assume $\|\varphi\|_2\le M$.
Recall $B_{q,n}$ and $\mathcal H_n$ from \eqref{sq_dec}.
  For $1\le q<d$, we have
\begin{align}\label{sq_low}
  \norm{B_{q,n}}_2 = \big|  \E[ g_n^2(\cZ) h_q(\cZ)] \big|
  \le C_q(1+\log n)^{q/2}.
\end{align}
Moreover,
\begin{align}\label{high_l2}
  \norm{\mathcal H_n}_2^2\le Cn.
\end{align}

\noi
Fix $0<\eps<1$ and put  $\dr=\frac{2\eps}{1+\eps}<1$.
Then,
there exists $C=C_{d,M,\eps}<\infty$ such that
for $|r|\le\eps$,

\noi
\begin{align}\label{sq_high}
  \sum_{q\ge d} \| B_{q,n}\|_2^2|r|^q
  \le Cn^\dr|r|^d.
\end{align}
The above constants are uniform over $\|\varphi\|_2\le M$ for each
fixed $M<\infty$.
\end{lemma}

\begin{proof}
Lemma~\ref{rearr_lem}\textup{(ii)}, applied to
$U=T_n(\cZ)^2=\varphi(\cZ)^2
\ind_{\{|\varphi(\cZ)|\le\sqrt n\}}$ with upper bound $\max\{n,2\}$, gives
\begin{align}\label{trunc_sq}
  \E[T_n(\cZ)^2|h_q(\cZ)|]
  \le C_q(1+\log n)^{\frac{q}{2}}.
\end{align}
Write $P_n=\sum_{j<d}\be_{j,n}h_j$ as in \eqref{def_PN}.
Then,    with $g_n = T_n - P_n$, we first write

\noi
\begin{align*}
\|B_{q,n}\|_2
&\le
\big|\E[T_n^2(\cZ)h_q(\cZ)]\big|
+\big|\E[P_n^2(\cZ)h_q(\cZ)]\big|
+2\big|\E[T_n(\cZ)P_n(\cZ)h_q(\cZ)]\big|,
\end{align*}
and thus deduce \eqref{sq_low}
from
\eqref{trunc_sq}, \eqref{bdd_PN},
and
\[
 \big|\E[T_n(\cZ)P_n(\cZ)h_q(\cZ)]\big|
 \leq \|T_n\|_2  \| P_n\|_4 \| h_q\|_4 \leq C_q.
\]
To prove \eqref{high_l2}, it suffices
to note that
$  \norm{\mathcal H_n}_2^2
  \le \Var(g_n(\cZ)^2)
  \le \E[g_n(\cZ)^4]
  \le Cn,$
by orthogonality \eqref{OL_pq} and  \eqref{g_mom}.

 \medskip

Finally, we prove \eqref{sq_high} by hypercontractivity of the
Ornstein--Uhlenbeck semigroup.  For $\rho\in[0,1)$ and
$F\in L^1(\ga)$, let
\begin{align}\label{OU_def}
 Q_\rho F(x)
 :=
 \int_{\R}
 F\bigl(\rho x+\sqrt{1-\rho^2}\,z\bigr)\,\ga(dz).
\end{align}
For $0<\rho<1$, $Q_\rho=P_t$ in the usual
Ornstein--Uhlenbeck notation when $\rho=e^{-t}$; see, e.g.,
\cite[Definition~1.3.1]{NP12}.  In particular,
\begin{align}\label{OU_chaos}
 Q_\rho h_q=\rho^q h_q,
 \quad q\ge0.
\end{align}
Equivalently, if $(X,Y_\rho)$ is a centered Gaussian vector with
standard normal marginals and correlation $\rho$, then
\begin{align}\label{OU_corr}
 \E[F(X)G(Y_\rho)]
 =
 \langle F,Q_\rho G\rangle_{L^2(\ga)}
\end{align}
for $F,G\in L^2(\ga)$.

Fix $\eps\in(0,1)$ as in the statement.  Nelson's
hypercontractivity inequality
\cite[Theorem~2.8.12]{NP12} implies that
\begin{align}\label{OU_hyper}
 Q_\eps:
 L^{1+\eps}(\ga)
 \longrightarrow
 L^{\frac{1+\eps}{\eps}}(\ga)
 \quad\text{is a contraction.}
\end{align}

Recall that
$ \dr=\frac{2\eps}{1+\eps}$,
so that
$ \frac{1}{2+2\eps}
 =
 \frac{1-\dr}{2}+\frac{\dr}{4}.$
 Then, we deduce from  \eqref{g_mom} and interpolation
 that

\noi
\begin{align}\label{interp}
 \|g_n\|_{2+2\eps}^4
 &\le
 \|g_n\|_2^{4(1-\dr)}
 \|g_n\|_4^{4\dr}
 \le C n^\dr.
\end{align}

We next apply \eqref{OU_hyper} to $g_n^2$.  By H\"older's inequality,
\eqref{OU_corr}, and hypercontractivity \eqref{OU_hyper},
\begin{align}
 \E[g_n(X)^2g_n(Y_\eps)^2]
 =
 \langle g_n^2,Q_\eps(g_n^2)\rangle_{L^2(\ga)}
 &\le
 \|g_n^2\|_{1+\eps}
 \|Q_\eps(g_n^2)\|_{\frac{1+\eps}{\eps}}
 \notag\\
 &\le
 \|g_n^2\|_{1+\eps}^2
 =
 \|g_n\|_{2+2\eps}^4
 \le Cn^\dr.
 \label{hyper_sq}
\end{align}

Recall $B_{q,n}$ is the $q$th Hermite-chaos projection of
$g_n^2$.  By \eqref{OU_chaos}, \eqref{hyper_sq}, and orthogonality
relation \eqref{OL_pq},
we can write, with $v_n = \E[ g_n^2(\cZ)]$,
\begin{align}\label{OU_series}
 \sum_{q\ge1}\eps^q\|B_{q,n}\|_2^2
 &=
 \big\langle g_n^2-v_n,
 Q_\eps(g_n^2-v_n)\big\rangle_{L^2(\ga)}
 \notag\\
 &\le
 \E[g_n(X)^2g_n(Y_\eps)^2]
 \le Cn^\dr.
\end{align}
Consequently, for $|r|\le\eps$,
\begin{align*}
 \sum_{q\ge d}\|B_{q,n}\|_2^2|r|^q
 &=
 \sum_{q\ge d}
 \|B_{q,n}\|_2^2\eps^q
 \left(\frac{|r|}{\eps}\right)^q
 \le
 \left(\frac{|r|}{\eps}\right)^d
 \sum_{q\ge d}\|B_{q,n}\|_2^2\eps^q
 \le
 C n^\dr |r|^d,
\end{align*}
where the factor $\eps^{-d}$ is absorbed into $C$.  This proves
\eqref{sq_high} and completes the proof.
\qedhere

\end{proof}

\subsection{An adjacent-block fourth-moment estimate}
\label{adj}

We now establish the mixed-moment gain in
Step~\eqref{step_adjacent}.

Let $Z_1,\ldots,Z_N$ be centered jointly Gaussian variables with
standard normal marginals.
Put $r(i,j)=\E[Z_iZ_j]$ and assume the conditions
\eqref{small_c1}, \eqref{small_c2}
with  $\mathcal I=\{1,\ldots,N\}$
and  $K_d\le K$.
Let $n$ be an integer in
 $[N, A_0N]$ for some fixed  $A_0\ge1$,
let $\varphi\in L^2(\ga)$ have Hermite rank at least $d$, and
retain
$g_n$ as in  \eqref{trunc_def}--\eqref{g_def}.

 For an integer interval
$I\subset\{1,\ldots,N\}$, put
\begin{align}\label{blk_sum}
  S_I=\sum_{i\in I}g_n(Z_i).
\end{align}

\begin{proposition}[\textsf{Adjacent-block estimate}]\label{adj_prop}
Under the preceding assumptions, suppose $\|\varphi\|_2\le M$.
Then the following statements hold.

\smallskip

{\rm (i)} $\E[S_I^2]\le KM^2|I|$ for every interval $I$.

\smallskip

{\rm (ii)}
There exist $\al \in(0, \frac{1}{4d}]$ and $C<\infty$, depending only on
$(d,\eps,K,A_0,  M)$
such that
\begin{align}\label{adj_bd}
 \E[S_I^2S_J^2]\le CN^{1-\al} ( |I| + |J| )^{1+\al}
\end{align}
for any adjacent disjoint integer intervals
$I,J\subset\{1,\ldots,N\}$.

\end{proposition}

\begin{proof}
Using the Hermite-rank condition and the orthogonality relation
\eqref{OL_pq},
we obtain
\begin{align}\label{use1}
 \big|\E[g_n(Z_i)g_n(Z_j)]\big|
 \le \|g_n\|_2^2 |r(i,j)|^d
 \le M^2 |r(i,j)|^d.
\end{align}
Consequently, by \eqref{small_c2},
\begin{align*}
 \E[S_I^2]
 =
 \sum_{i,j\in I}\E[g_n(Z_i)g_n(Z_j)]
  \le
 M^2\sum_{i,j\in I}|r(i,j)|^d
 \le K M^2 |I|,
\end{align*}
which proves part~\textup{(i)}.

We now prove \eqref{adj_bd}.  Put $m:=|I|+|J|$.
Expanding $\E[S_I^2S_J^2]$ gives a sum over
$(i,i',j,j')\in I^2\times J^2$. Since $I\cap J=\varnothing$,
there are three cases up to symmetry.

\medskip

$\bullet$ {\bf Case 1.}
Suppose $i\ne i'$ and $j\ne j'$. Since $I\cap J=\varnothing$,
the four indices $i,i',j,j'$ are all distinct. By
Lemma~\ref{diag_lem}\textup{(i)} and
$\|g_n\|_2\le\|\varphi\|_2\le M$, the total contribution of these
terms is bounded by
$
 Cm^2.
$

\medskip

$\bullet$ {\bf Case 2.}
Suppose $i=i'$ and $j\ne j'$. The case $i\ne i'$ and $j=j'$
is treated identically. By \eqref{sq_dec},
\begin{align*}
 &\E[g_n(Z_i)^2g_n(Z_j)g_n(Z_{j'})]\\
 &\quad=
 \|g_n\|_2^2\E[g_n(Z_j)g_n(Z_{j'})]
 + \bigg( \sum_{q=1}^{d-1}
   \E[B_{q,n}(Z_i)g_n(Z_j)g_n(Z_{j'})] \bigg)
 +\E[\mathcal H_n(Z_i)g_n(Z_j)g_n(Z_{j'})].
\end{align*}

For the first term, \eqref{use1} and \eqref{small_c2} give
\begin{align*}
 \sum_{i\in I}
   \sum_{\substack{j,j'\in J\\ j\ne j'}}
 \|g_n\|_2^2
 \big|\E[g_n(Z_j)g_n(Z_{j'})]\big|
  \le C | I |
 \sum_{j\in J}\sum_{j'\in J}|r(j,j')|^d
 \le C|I||J|
 \le Cm^2.
\end{align*}

For the low-chaos terms, Lemma~\ref{diag_lem}\textup{(iii)}
together with \eqref{sq_low} yields
\[
 C\sum_{q=1}^{d-1}
 (1+\log n)^{\frac{q}{2}}m^{2-\frac{q}{2d}}.
\]
For the high-chaos term, Lemma~\ref{diag_lem}\textup{(ii)}
together with \eqref{high_l2} gives
$
 C\sqrt n\,m^{\frac{3}{2}}.
$
Thus the contribution from {\bf Case 2}
is bounded by
\begin{align}\label{one_pair}
 Cm^2
 +C\sum_{q=1}^{d-1}
 (1+\log n)^{\frac{q}{2}}m^{2-\frac{q}{2d}}
 +C\sqrt n\,m^{\frac{3}{2}}.
\end{align}

\medskip

$\bullet$ {\bf Case 3.}
Suppose $i=i'$ and $j=j'$. Since $I\cap J=\varnothing$, we have
$i\ne j$. Let
$
 v_n=\E[g_n(\cZ)^2].
$
By orthogonality \eqref{OL_pq} and \eqref{sq_dec},
\begin{align}\label{sq_mehler}
 \E[g_n(Z_i)^2g_n(Z_j)^2]
 =
 v_n^2+\sum_{q\ge1}\|B_{q,n}\|_2^2 r(i,j)^q.
\end{align}
Since $v_n\le M^2$, the constant term in \eqref{sq_mehler},
summed over $I\times J$, contributes at most
$ Cm^2$.

We next treat the low-chaos terms. By \eqref{small_c2},
\begin{align}\label{pair_ld}
 \sum_{i\in I,j\in J}|r(i,j)|^d
 &\le
 K\min\{|I|,|J|\}
 \le Cm.
\end{align}
Hence, for $1\le q<d$, H\"older's inequality gives
\begin{align}
 \sum_{i\in I,j\in J}|r(i,j)|^q
 &\le
 (|I||J|)^{1-\frac{q}{d}}
 \left(
   \sum_{i\in I,j\in J}|r(i,j)|^d
 \right)^{q/d}
 \le
 Cm^{2-\frac{q}{d}}.
 \label{corr_low}
\end{align}
Combining \eqref{corr_low} with \eqref{sq_low}, the total
contribution from the orders $1\le q<d$ is bounded by
\[
 C\sum_{q=1}^{d-1}
 (1+\log n)^q m^{2-\frac{q}{d}}.
\]

For $q\ge d$, the small-correlation condition \eqref{small_c1}
gives
$ |r(i,j)|\le\eps$
for $ (i,j)\in I\times J.$
Hence, for each such pair, \eqref{sq_high} yields
\[
 \sum_{q\ge d}
 \|B_{q,n}\|_2^2|r(i,j)|^q
 \le
 Cn^\dr |r(i,j)|^d,
\]
where
 $\dr=\frac{2\eps}{1+\eps}$,
and $C=C_{d,M,\eps}$ is independent of $n,i,j$. Summing over
$I\times J$ and using \eqref{small_c2}, we obtain
\begin{align*}
 \sum_{i\in I,j\in J}
 \sum_{q\ge d}
 \|B_{q,n}\|_2^2|r(i,j)|^q
 &\le
 Cn^\dr
 \sum_{i\in I,j\in J}|r(i,j)|^d
 \le
 Cn^\dr |I|
 \le
 Cn^\dr m.
\end{align*}
Therefore, the total contribution from {\bf Case 3} is bounded by
\begin{align}\label{two_pair}
 Cm^2
 +C\sum_{q=1}^{d-1}
 (1+\log n)^q m^{2-\frac{q}{d}}
 +Cn^\dr m.
\end{align}

It remains to compare the bounds above with the right-hand side of
\eqref{adj_bd}.  It suffices to take $\al=\frac{1}{4d}$, since the
right-hand side is nonincreasing in $\al$ for $m\le N$.
We have $\dr=\frac{2\eps}{1+\eps}<\frac{1}{2}$, $n\le A_0N$, and $1\le
m\le N$.
Consequently,
\begin{align*}
 m^2+\sqrt n\,m^{\frac{3}{2}}+n^\dr m
 \le C N^{1-\al}m^{1+\al}.
\end{align*}
For $1\le q<d$, both $1-\al- \frac{q}{2d}$ and $1-\al- \frac{q}{d} $ are
positive,
so the remaining terms satisfy
\begin{align*}
 \frac{(1+\log n)^{\frac{q}{2}}m^{2- \frac{q}{2d} }}{N^{1-\al}m^{1+\al}}
 &\le C(1+\log N)^{\frac{q}{2}}N^{-\frac{q}{2d}}\le C,\\
 \frac{(1+\log n)^q m^{2-\frac{q}{d} }}{N^{1-\al}m^{1+\al}}
 &\le C(1+\log N)^qN^{-\frac{q}{d}}\le C.
\end{align*}
Combining {\bf Cases 1--3} proves \eqref{adj_bd}.
Hence, the proof is complete.
\qedhere

\end{proof}

\subsection{Restoring small-support terms in \texorpdfstring{$L^2$}{L2}}
\label{small_support_sec}

For Step~\eqref{step_restore}, the following lemma restores the
discarded tail in the same $L^2$ norm as the truncated maximal
bound.  Taking an absolute value destroys Hermite rank, but
small support makes the low Hermite coefficients small enough.

\begin{lemma}[\textsf{Small-support estimate}]
\label{small_support}
Assume \eqref{ns_rows}.  Let $u\in L^2(\ga)$ vanish outside a
measurable set $A$ with $\ga(A)\leq 1/ n$.  Then
\begin{align}\label{small_support_bd}
 \E\bigg[ \Big(\sum_{k=1}^n|u(Z_k)|\Big)^2 \bigg]
 \le C_{d,K}n\|u\|_2^2.
\end{align}
In particular, 
\begin{align}\label{bd_MA}
\E\bigg[ \max_{m\le n} \Big|\sum_{k=1}^m u(Z_k)\Big|^2 \bigg]
 \leq  C_{d,K}n\|u\|_2^2.
 \end{align}
\end{lemma}

\begin{proof}
Since $u\in L^2(\ga)$, we can expand  $|u|=\sum_{q\ge0}b_qh_q$ in $L^2(\ga)$.
Put
$r_{ij} = \E[ Z_i Z_j]$.  Then, we can first deduce from 
 the orthogonality relation \eqref{OL_pq}
that

\noi
\begin{align}\label{support_split}
 \E\left(\sum_{k=1}^n |u| (Z_k)\right)^2
 &\le n^2b_0^2
   +\sum_{q=1}^{d-1}b_q^2\sum_{i,j=1}^n|r_{ij}|^q
   +\sum_{q\ge d}b_q^2\sum_{i,j=1}^n|r_{ij}|^q.
\end{align}

\noi
For the constant term,
\begin{align*}
 b_0^2=\left(\int_A |u|\,d\ga\right)^2
 \le \ga(A) \|u\|_2^2\le n^{-1}\|u\|_2^2.
\end{align*}
For $q\ge d$, the row bound \eqref{ns_rows} gives
\begin{align*}
 \sum_{i,j=1}^n|r_{ij}|^q\le Kn,
\end{align*}
so the corresponding part of \eqref{support_split} is at most
$Kn\|u\|_2^2$.

If $d\ge2$ and $1\le q<d$, H\"older's inequality gives
\begin{align} \label{small_cal}
\begin{aligned}
 \sum_{i,j=1}^n|r_{ij}|^q
 \le K^{\frac{q}{d}}n^{2-\frac{q}{d}}
 \AND
 b_q^2
 \le\|u\|_2^2\int_A h_q^2\,d\ga
 &\le\|u\|_2^2\|h_q\|_{2d}^2 \ga(A)^{1-\frac{1}{d}} \\
 &\le C_d\|u\|_2^2n^{-1+\frac{1}{d}}.
\end{aligned}
\end{align}
Consequently,
\begin{align*}
 b_q^2\sum_{i,j=1}^n|r_{ij}|^q
 \le C_{d,K}\|u\|_2^2n^{1+\frac{1-q}{d}}
 \le C_{d,K}n\|u\|_2^2.
\end{align*}
There are only finitely many such orders.  For $d=1$, the low-order
sum is empty.  
Therefore, combining the preceding estimates proves  \eqref{small_support_bd}.  
Finally, 
  from $\max_{m\le n}\left|\sum_{k=1}^m u(Z_k)\right|
 \le\sum_{k=1}^n|u(Z_k)|$,
 we 
prove  \eqref{bd_MA}.
\qedhere

\end{proof}

\subsection{Proof of the endpoint maximal inequality}
\label{max_pf}

\begin{proof}[\textbf{Proof of Theorem~\ref{ns_max} for $p=2$}]
We first prove \eqref{ns_l2} when distinct coordinates have
small correlations, then remove this restriction by a partition
argument.  The cases $f=0$ and $n=1$ are immediate.
By homogeneity, we may otherwise assume $\|f\|_2=1$ and $n\ge2$.

\smallskip
$\bullet$  \textbf{Separated case.}
First assume\footnote{The choice $\dl=\frac{1}{8}$ is a
convenient fixed small number less than $\frac{1}{3}$,
as required by \eqref{small_c1}.}
\begin{align}\label{assump_31}
 |\E[Z_iZ_j]|\le\dl=\frac{1}{8},
 \quad i\ne j.
\end{align}
As in \eqref{trunc_def}--\eqref{g_def}, define
\begin{align*}
 T_n=f\ind_{\{|f|\le\sqrt n\}}
 \AND
 g_n=T_n-\sum_{q<d}\be_{q,n}h_q,
\end{align*}
where $\be_{q,n}=\langle T_n,h_q\rangle_{L^2(\ga)}$.
The estimates in Sections~\ref{clip}--\ref{square} give
\begin{align}\label{ns_clip}
 \begin{aligned}
 &\|g_n\|_2\le1,\quad \|g_n\|_4^4\le C_dn,
   \quad |\be_{0,n}|\le n^{-\frac{1}{2}},\\
 &|\be_{q,n}|
   \le C_qn^{-\frac{1}{2}}(1+\log n)^{\frac{q}{2}},
   \quad 1\le q<d.
 \end{aligned}
\end{align}
The Gaussian vector $(Z_1,\ldots,Z_n)$ satisfies the hypotheses
of Proposition~\ref{adj_prop} with $N=n$, $A_0=M=1$,
and $\eps=\dl = 1/8$.
Thus Proposition~\ref{adj_prop} and Lemma~\ref{max_lem} yield
\begin{align}\label{ns_gmax}
 \E\max_{0\le m\le n}
 \left|\sum_{k=1}^m g_n(Z_k)\right|^2
 \le C_{d,K}n.
\end{align}

\smallskip
$\bullet$
\textbf{Restoration in $L^2$: Step~\eqref{step_restore}.}
  Put
$E_n=\{|f|>\sqrt n\}$ and
$\tau_n=\|f\ind_{E_n}\|_2\le1$.
Since $\ga(E_n)\leq 1/n$,
Lemma~\ref{small_support} gives
\begin{align}\label{ns_tail}
 \left\|\max_{0\le m\le n}
 \left|\sum_{k=1}^m(f-T_n)(Z_k)\right|\right\|_2
 \le C_{d,K}\sqrt n\,\tau_n.
\end{align}
Next, let $P_n=\sum_{q<d}\be_{q,n}h_q$ as before.
We claim that
\begin{align}\label{12d}
 \left\|\max_{0\le m\le n}
 \left|\sum_{k=1}^mP_n(Z_k)\right|\right\|_2
 \le C_{d,K}\sqrt n\,\tau_n.
\end{align}
Since the low Hermite coefficients of $f$ vanish,
$\be_{q,n}=-\langle f\ind_{E_n},h_q\rangle$.
Thus $n|\be_{0,n}|\le\sqrt n\,\tau_n$.
For $d\ge2$ and $1\le q<d$, the small-set calculation \eqref{small_cal}
in the proof of Lemma~\ref{small_support} gives
\begin{align}\label{restore_coef}
 |\be_{q,n}|
 \le\tau_n
 \left(\int_{E_n}h_q^2\,d\ga\right)^{\frac{1}{2}}
  \leq \tau_n  \| h_q\|_{2d} \, \ga(E_n)^{\frac{1}{2} - \frac{1}{2d}}
 \le C_d\tau_n n^{-\frac{1}{2}+\frac{1}{2d}}.
\end{align}
It remains to combine this with an $L^2$ maximal estimate
for   fixed   $1\leq q < d$.

Fix such a $q$.  By \eqref{OL_pq}, for every integer interval
$I\subset\{1,\ldots,n\}$ of cardinality $m$,
\begin{align}\label{ns_low}
 \E\left|\sum_{k\in I}h_q(Z_k)\right|^2
 \le\sum_{i,j\in I}|\E[Z_iZ_j]|^q
 \le K^{\frac{q}{d}}m^{2-\frac{q}{d}},
\end{align}
where the last inequality follows from H\"older's inequality
as in \eqref{corr_low}.
We now apply Lemma~\ref{com_lem} to obtain the required
low-chaos maximal estimate.
Let $n'$ be the smallest power of two with $n'\ge n$,
so that $n\le n'<2n$, and set
\begin{align*}
 X_k=
 \begin{cases}
  h_q(Z_k),&1\le k\le n,\\
  0,&n<k\le n'.
 \end{cases}
\end{align*}
Using the dyadic partitions $\mathcal D_\ell$ from
Lemma~\ref{com_lem} with $N=n'$,
we can obtain 

\noi
\begin{align}\label{use2a2}
 \max_{0\le m\le n}
 \left|\sum_{k=1}^m h_q(Z_k)\right|
 \le
 \sum_{\substack{\ell\le n'\\ \ell\ {\rm dyadic}}}
 \max_{D\in\mathcal D_\ell}
  \Big| \sum_{k\in D}X_k   \Big|.
\end{align}
For each dyadic length $\ell$, there are $n' /\ell$
blocks in $\mathcal D_\ell$.
Since $D\cap\{1,\ldots,n\}$ is empty or an integer interval
of length at most $\ell$, \eqref{ns_low} yields
\begin{align*}
 \left\|
 \max_{D\in\mathcal D_\ell}
 \Big|  \sum_{k\in D}X_k \Big|
 \right\|_2^2
 &\le
 \sum_{D\in\mathcal D_\ell}
 \left\|\sum_{k\in D}X_k\right\|_2^2\\
 &\le
 K^{\frac{q}{d}}\frac{n'}{\ell}\,
 \ell^{2-\frac{q}{d}}
 \le Cn\,\ell^{1-\frac{q}{d}}.
\end{align*}
Here the first inequality follows from
$\max_a|U_a|^2\le\sum_a|U_a|^2$, and the last uses $n'<2n$.
Consequently,
\begin{align}\label{use2b}
 \left\|
 \max_{D\in\mathcal D_\ell}
 \Big|  \sum_{k\in D}X_k  \Big|
 \right\|_2
 \le C\sqrt n\,\ell^{\frac{1}{2}(1-\frac{q}{d})}.
\end{align}
Combining \eqref{use2a2} and \eqref{use2b}
with Minkowski's inequality gives
\begin{align}\label{ns_lowmax}
 \begin{aligned}
 \left\|  \max_{0\le m\le n}
 \Big|  \sum_{k=1}^m h_q(Z_k) \Big|   \right\|_2
 \leq
 C\sqrt n
 \sum_{\substack{\ell\le n'\\ \ell\ {\rm dyadic}}}
 \ell^{\frac{1}{2}(1-\frac{q}{d})} 
& \leq
 C_{d,K,q}\sqrt n\,
 (n')^{\frac{1}{2}(1-\frac{q}{d})} \\
 &\le C_{d,K,q}n^{1-\frac{q}{2d}}.  
 \end{aligned}
 \end{align}
The second inequality uses $q<d$, so the dyadic sum is
bounded by a constant times its largest term.

Combining \eqref{restore_coef} and \eqref{ns_lowmax},
we obtain
\begin{align*}
 |\be_{q,n}|
 \left\|\max_{0\le m\le n}
 \Big|   \sum_{k=1}^m h_q(Z_k)  \Big|  \right\|_2
 &\le
 C\tau_n n^{-\frac{1}{2}+\frac{1}{2d}}
 n^{1-\frac{q}{2d}}\\
 &=C\tau_n n^{\frac{1}{2}-\frac{q-1}{2d}}
 \le C\sqrt n\,\tau_n.
\end{align*}
Summing over $1\le q<d$ and including the constant
correction proves \eqref{12d}.
When $d=1$, only the constant correction is present.
Since $f=g_n+(f-T_n)+P_n$, the  triangle inequality
and Minkowski's inequality, together with \eqref{ns_gmax},
\eqref{ns_tail}, and \eqref{12d}, give
\begin{align*}
 \left\|\max_{0\le m\le n}
 \left|\sum_{k=1}^m f(Z_k)\right|\right\|_2
 \le C_{d,K}\sqrt n(1+\tau_n)
 \le C_{d,K}\sqrt n.
\end{align*}
This proves \eqref{ns_l2} under \eqref{assump_31}.
%
%

\medskip
$\bullet$
\textbf{Finite partition and recombination:
Step~\eqref{step_partition}.}
We now remove \eqref{assump_31}.
Construct a graph on $\{1,\ldots,n\}$ by joining two
distinct indices $i$ and $j$ whenever
$|\E[Z_iZ_j]|>\dl=\frac{1}{8}$.
By \eqref{ns_rows}, for every $i$,
\begin{align*}
 K
 \ge\sum_{j=1}^n|\E[Z_iZ_j]|^d
 \ge\dl^d\,
 \#\bigl\{j\ne i:\ |\E[Z_iZ_j]|>\dl\bigr\}.
\end{align*}
Every vertex therefore has degree at most
$K\dl^{-d}=K8^d$.
Then, we can  partition  $\{1,\ldots,n\}$ into at most
$L:=1+\lfloor K8^d\rfloor$ classes, say $C_1,\ldots,C_L$,
such that $|\E[Z_iZ_j]|\le\dl$ for distinct $i,j\in C_\ell$.
Each class inherits \eqref{ns_rows} with the same $K$,
so the separated-coordinate estimate applies within each class.

For each $\ell$, write
$C_\ell=\{k_{\ell,1}<\cdots<k_{\ell,n_\ell}\}$
in the order inherited from $\{1,\ldots,n\}$.
For every $m\in\{0,\ldots,n\}$, the indices of $C_\ell$
lying in $\{1,\ldots,m\}$ form an initial segment
$\{k_{\ell,1},\ldots,k_{\ell,r_\ell(m)}\}$
for some $0\le r_\ell(m)\le n_\ell$.
Hence
\begin{align*}
 \max_{0\le m\le n}
 \left|\sum_{k=1}^m f(Z_k)\right|
 &\le
 \sum_{\ell=1}^L
 \max_{0\le r\le n_\ell}
 \left|\sum_{a=1}^r f(Z_{k_{\ell,a}})\right|.
\end{align*}
Applying the separated-coordinate $L^2$ estimate to each
class, with its own length $n_\ell$, and using
Minkowski's inequality, we obtain
\begin{align*}
 \left\|\max_{0\le m\le n}
 \left|\sum_{k=1}^m f(Z_k)\right|\right\|_2
 \le C_{d,K}\sum_{\ell=1}^L\sqrt{n_\ell}
 \le C_{d,K}\sqrt{L\sum_{\ell=1}^L n_\ell}
 =C_{d,K}\sqrt{Ln}.
\end{align*}
Since $L=1+\lfloor K8^d\rfloor$ depends only on $d$ and $K$,
the factor $\sqrt L$ can be absorbed into the constant.
Hence, the proof is complete.
\qedhere

\end{proof}

\section{Stationary BMD principles}
\label{stat_pf}

We now apply the endpoint $p=2$ of Theorem~\ref{ns_max} and
Proposition~\ref{ns_int} to the stationary BMD principles.
For finite Hermite expansions, the Breuer--Major theorems give
finite-dimensional convergence, while covariance bounds and
hypercontractivity give tightness.  The maximal
inequality then controls the Hermite tails uniformly in the
$L^2$ path norm and passes convergence to a general
$\varphi\in L^2(\ga)$.

\subsection{The discrete-time BMD}
\label{disc_pf}

\begin{proof}[Proof of Corollary~\ref{disc_thm}]
Put
$K=\sum_{k\in\Z}|\rho_G(k)|^d$,
and, for $M\ge d$, define
\begin{align}\label{herm_cut}
 \varphi_M=\sum_{q=d}^M c_qH_q.
\end{align}
We write $V_n^f$ and $\wt{V_n^f}$ for \eqref{d_V} and
\eqref{hat_V}, respectively, with $f$ in place of $\varphi$.

We first prove the functional convergence for the step process
associated with the finite Hermite expansion $\varphi_M$.
For $0\le s<t\le1$, put
$I_{s,t}^{(n)}=\{\lfloor ns\rfloor+1,\ldots,\lfloor nt\rfloor\}$.
Using the orthogonality relation and $|\rho_G|\le1$, we obtain
\begin{align}
 \E\left|
 V_n^{\varphi_M}(t)-V_n^{\varphi_M}(s)
 \right|^2
&=
 \frac{1}{n}
 \sum_{i,j\in I_{s,t}^{(n)}}
 \sum_{q=d}^M q!c_q^2\rho_G(i-j)^q
 \notag\\
 &\quad\le
 \frac{\|\varphi_M\|_2^2}{n}
 \sum_{i,j\in I_{s,t}^{(n)}}
 |\rho_G(i-j)|^d
\leq
 K\|\varphi\|_2^2
 \frac{\lfloor nt\rfloor-\lfloor ns\rfloor}{n}.
 \label{disc_trunc_l2}
\end{align}
Moreover,
$V_n^{\varphi_M}(t)-V_n^{\varphi_M}(s)$
belongs to the sum of the first $M$ Wiener chaoses.
Hence, for every fixed $p>2$, hypercontractivity
\cite[Theorem~2.8.12]{NP12} and \eqref{disc_trunc_l2} yield
\begin{align}\label{disc_trunc_lp}
 \left\|
 V_n^{\varphi_M}(t)-V_n^{\varphi_M}(s)
 \right\|_p
 \le
 C_{p,M,K}\|\varphi\|_2
 \left(
 \frac{\lfloor nt\rfloor-\lfloor ns\rfloor}{n}
 \right)^{\frac{1}{2}}.
\end{align}
The standard tightness criterion for partial-sum processes
(see, e.g., \cite[Lemma~3.1]{NN20}) therefore implies that
$\{V_n^{\varphi_M}\}_{n\ge1}$ is tight in $\mathbf D([0,1])$.

On the other hand, Theorem~\ref{bm_thm} gives the convergence of the
finite-dimensional distributions,
$ V_n^{\varphi_M}
 \stackrel{\rm f.d.d.}{\Longrightarrow}
 \s_{G,M}\cW,
$
where
$ \s_{G,M}^2
 =
 \sum_{q=d}^M q!c_q^2
 \sum_{k\in\Z}\rho_G(k)^q.
$
Thus, we have
\begin{align}\label{disc_trunc}
 V_n^{\varphi_M}
 \LRA
 \s_{G,M}\cW
 \quad\text{in }\mathbf D([0,1]).
\end{align}
In particular, $\{V_n^{\varphi_M}\}_{n\ge1}$ is $C$-tight.

We now remove the Hermite truncation.  Since
$\varphi-\varphi_M$ has Hermite rank at least $d$, every finite
segment of $\mathbf G$ satisfies the assumptions of
Theorem~\ref{ns_max} with the same row-sum constant $K$.
Consequently,
\begin{align}\label{disc_approx}
 \sup_{n\ge1}
 \E \bigg[ \|V_n^\varphi-V_n^{\varphi_M}\|_\infty^2 \bigg]
 &=\sup_{n\ge1}\frac{1}{n}
   \E \bigg( \max_{0\le m\le n}
   \Big|\sum_{k=1}^m
         (\varphi-\varphi_M)(G_k)  \Big|^2 \bigg)  \notag\\
 &\le C_{d,K}\|\varphi-\varphi_M\|_2^2
 \longrightarrow0
\end{align}
as $M\to\infty$.
Moreover,
\begin{align*}
 |\s_G^2-\s_{G,M}^2|
 &\le
 \sum_{q>M}q!c_q^2
 \sum_{k\in\Z}|\rho_G(k)|^q
 \le
 K\sum_{q>M}q!c_q^2
 =
 K\|\varphi-\varphi_M\|_2^2
 \longrightarrow0.
\end{align*}
Combining \eqref{disc_trunc}, \eqref{disc_approx}, and
$\s_{G,M}\to\s_G$ with Billingsley's theorem
\cite[Theorem~3.2]{Bil99}, we conclude that
$V_n^\varphi
 \LRA
 \s_G\cW$
 in $\mathbf{D}([0,1])$.
Since the limit is continuous, the sequence
$\{V_n^\varphi\}_{n\ge1}$ is $C$-tight.
Finally, Proposition~1.7 of \cite{MPZ26} shows that convergence of
the step partial-sum process in $\mathbf D([0,1])$ to a continuous
limit is equivalent to convergence of its polygonal interpolation in
$\mathbf C([0,1])$.  As a result, we have
$\wt{V_n^\varphi}  \LRA \s_G\cW$
in $\mathbf{C}([0,1])$.
This completes the proof.
\end{proof}

\subsection{The continuous-time BMD}
\label{cont_pf}

We first justify the continuity of $\rho_B$ asserted in the
introduction.  Joint measurability gives
$B_\cdot\in L^2_{\rm loc}(\R)$ almost surely.  Continuity of
translations in $L^2$ therefore implies
\begin{align*}
 \int_0^1|B_{x+h}-B_x|^2\,dx\longrightarrow0
 \quad\text{almost surely as }h\to0.
\end{align*}
These integrals are bounded in $L^2(\Omega)$, since
$\E|B_{x+h}-B_x|^4\le48$.
Uniform integrability, Fubini's theorem, and stationarity yield
$\rho_B(h)\to1$.  The covariance identity then gives continuity
of $\rho_B$ and mean-square continuity of $\mathbf B$.

The continuous-time BMD principle is obtained from the discrete theory
through a simple transfer mechanism.  The key input is the maximal
inequality of Theorem~\ref{ns_max}.  After choosing a lattice
spacing for which the sampled covariance is $\ell^d$-summable, we
decompose the time integral into lattice fibers and apply the discrete
maximal estimate uniformly along these fibers.  Proposition~\ref{ns_int}
then provides the uniform approximation bound needed to pass from
finite Hermite expansions to a general function in $L^2(\ga)$.

\begin{lemma}[\textsf{Almost-everywhere lattice summability}]
\label{lat_lem}
Let $h:\R\to[0,\infty)$ be measurable with $h\in L^1(\R)$ and
$h(0)<\infty$.
Then, for every \(A>0\),
\begin{align}\label{lat_int}
 \int_A^{2A}\sum_{k\in\Z}h(ak)\,da<\infty.
\end{align}
Consequently, \(\sum_{k\in\Z}h(ak)<\infty\) for almost every
\(a\in[A,2A]\).
\end{lemma}

\begin{proof}
The term
$k=0$ contributes $Ah(0)$.  For positive $k$,
by Tonelli's theorem, we have
\begin{align*}
 \int_A^{2A}\sum_{k=1}^{\infty}h(ak)\,da
 &=\sum_{k=1}^{\infty}\frac{1}{k}\int_{Ak}^{2Ak}h(x)\,dx\\
 &=\int_A^\infty h(x)
   \sum_{\frac{x}{2A}\le k\le \frac{x}{A}}\frac{1}{k}\,dx
 \leq (1 + \log 2)\int_A^\infty h(x)\,dx.
\end{align*}
The negative
indices are treated in the same way with \(h(-x)\).  This proves the
lemma.
\end{proof}

We now prove  Proposition~\ref{ns_int} by integrating the maximal bound
over the unit-lattice fibers and controlling incomplete unit intervals.

\begin{proof}[Proof of Proposition~\ref{ns_int}]
By joint measurability, Fubini,  and the standard
normal marginals of $Y$, we have for every finite $R>0$,
\begin{align}\label{01_bdd}
 \E\int_0^R|g(Y(s))|\,ds
 =R\|g\|_1<\infty.
\end{align}
Thus the integrand is locally integrable almost surely,
and the integral paths are continuous.
Moreover, Minkowski's inequality gives
\begin{align}\label{cell_l2}
 \left\|\int_a^b|g(Y(s))|\,ds\right\|_2
 \le\int_a^b\|g(Y(s))\|_2\,ds
 =(b-a)\|g\|_2,
 \quad 0\le a\le b.
\end{align}
If $\eps=1$, then
\begin{align*}
 \sup_{0\le t\le1}|F_1^g(t)|
 =
 \sup_{0\le t\le1}
 \left|\int_0^t g(Y(s))\,ds\right|
 \le\int_0^1|g(Y(s))|\,ds.
\end{align*}
Hence \eqref{cell_l2} yields
\begin{align*}
 \left\|\sup_{0\le t\le1}|F_1^g(t)|\right\|_2
 \le\left\|\int_0^1|g(Y(s))|\,ds\right\|_2
 \le\|g\|_2.
\end{align*}

Assume henceforth that $0<\eps<1$ and put
$N=\lceil1/\eps\rceil-1\ge1$.
For every integer $m\ge0$, decomposing $[1,1+m]$
into unit intervals gives the fiber identity
\begin{align*}
 \int_1^{1+m}g(Y(s))\,ds
 &=
 \sum_{k=0}^{m-1}\int_0^1g(Y(1+u+k))\,du
 =
 \int_0^1\sum_{k=0}^{m-1}g(Y(1+u+k))\,du,
\end{align*}
with the empty sum interpreted as zero.
For each $u\in[0,1]$, the row bound \eqref{ns_fib}
allows us to apply Theorem~\ref{ns_max} with $p=2$
to the corresponding Gaussian fiber.
Taking the maximum inside the integral and then applying
Minkowski's inequality therefore yields
\begin{align}  \label{ns_end}
\begin{aligned}
 \left\|\max_{0\le m\le N}
 \Big|\int_1^{1+m}g(Y(s))\,ds\Big|\right\|_2
 &\le
 \int_0^1
 \left\|\max_{0\le m\le N}
 \Big| \sum_{k=0}^{m-1}g(Y(1+u+k)) \Big|
 \right\|_2\,du \\
 &\quad\le C_{d,K}\sqrt N\,\|g\|_2.
\end{aligned}
\end{align}
Set
$A_k=\int_{1+k}^{2+k}|g(Y(s))|\,ds$ for $0\le k<N$.
By \eqref{cell_l2},
\begin{align}\label{12_bdd}
 \left\|\max_{0\le k<N}A_k\right\|_2
 \le
 \left(\sum_{k=0}^{N-1}\E[A_k^2]\right)^{\frac{1}{2}}
 \le\sqrt N\,\|g\|_2.
\end{align}

To pass from integer endpoints to arbitrary times,
fix $t\in[0,1]$ and put $r=t/\eps$.
If $r\le1$, then
\begin{align*}
 \left|\int_0^r g(Y(s))\,ds\right|
 \le\int_0^1|g(Y(s))|\,ds.
\end{align*}
If $r>1$, set
$m=\lfloor r\rfloor-1$
and
$u=r-(1+m)\in[0,1)$.
Then $r=1+m+u$ and
\begin{align}\label{cell_dec}
 \int_0^r g(Y(s))\,ds
 &=
 \int_0^1g(Y(s))\,ds
 +\int_1^{1+m}g(Y(s))\,ds
+
 \int_{1+m}^{1+m+u}g(Y(s))\,ds.
\end{align}
This separates the initial unit cell, the complete unit cells
ending at $1+m$, and the remaining part of the last cell;
see Figure~\ref{fig_fiber_cells}.
Since $r\le1/\eps\le N+1$, we have $0\le m\le N$.
If $u>0$, necessarily $m<N$, and
\begin{align*}
 \left|\int_{1+m}^{1+m+u}g(Y(s))\,ds\right|
 \le\int_{1+m}^{2+m}|g(Y(s))|\,ds
 =A_m.
\end{align*}
If $u=0$, the last integral in \eqref{cell_dec} vanishes.

\begin{figure}[t]
\centering
\begin{tikzpicture}[
    x=2cm,
    y=.85cm,
    every node/.style={font=\small},
    >=Latex
]

\fill[gray!15] (0,0) rectangle (1,.32);
\fill[gray!8] (1,0) rectangle (4,.32);
\fill[gray!30] (4,0) rectangle (4.65,.32);

\draw[line width=.8pt] (0,0) -- (5.15,0);
\draw[line width=.8pt] (0,.32) -- (4.65,.32);
\foreach \x in {0,1,2,3,4}{
    \draw (\x,-.07) -- (\x,.39);
}
\draw (4.65,0) -- (4.65,.39);
\draw[dashed] (4.65,.32) -- (5,.32) -- (5,0);
\draw (5,-.07) -- (5,.07);

\node[below=3pt] at (0,0) {$0$};
\node[below=3pt] at (1,0) {$1$};
\node[below=3pt] at (2.5,0) {$\cdots$};
\node[below=3pt] at (4,0) {$1+m$};
\node[below=3pt] at (5,0) {$2+m$};

\draw[decorate,decoration={brace,amplitude=4pt}]
    (0,.65) -- (1,.65)
    node[midway,above=6pt] {initial cell};

\draw[decorate,decoration={brace,amplitude=4pt}]
    (1,.65) -- (4,.65)
    node[midway,above=6pt] {complete unit cells};

\draw[decorate,decoration={brace,amplitude=4pt}]
    (4,.65) -- (4.65,.65)
    node[midway,above=6pt] {incomplete cell};

\draw[->,line width=.8pt]
    (4.65,-.95) -- (4.65,-.12);
\node[below=3pt] at (4.65,-.95)
    {$r=t/\eps=1+m+u$};

\end{tikzpicture}
\caption{\small Decomposition of $[0,r]$ when $r>1$.
The incomplete part $[1+m,r]$ lies in the unit cell
$[1+m,2+m]$ and is absent when $u=0$.}
\label{fig_fiber_cells}
\end{figure}
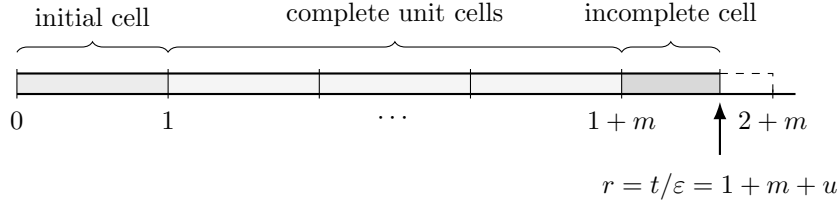

Taking the supremum over $t\in[0,1]$ gives
\begin{align*}
 \sup_{0\le t\le1}
 \left|\int_0^{t/\eps}g(Y(s))\,ds\right|
 &\le\int_0^1|g(Y(s))|\,ds
 +
 \max_{0\le m\le N}
 \left|\int_1^{1+m}g(Y(s))\,ds\right| 
 +\max_{0\le k<N}A_k.
\end{align*}
Multiplying by $\sqrt\eps$ and using Minkowski's inequality,
\eqref{cell_l2}, \eqref{ns_end}, and \eqref{12_bdd}, we obtain
\begin{align*}
 \left\|\sup_{0\le t\le1}|F_\eps^g(t)|\right\|_2
 &\le
 \bigl(\sqrt\eps+C_{d,K}\sqrt{\eps N}\bigr)\|g\|_2
 \le C_{d,K}\|g\|_2,
\end{align*}
since $\eps N\le1$ and $\eps\le1$.
Hence,  \eqref{ns_intbd} follows immediately
and 
the assertion for differences follows by linearity.
\qedhere

\end{proof}

Now we are ready to prove the continuous-time BMD principle 
in the stationary setting.

\begin{proof}[Proof of Corollary~\ref{cont_thm}]
Write $J_d=\int_\R|\rho_B(x)|^d\,dx$.
Apply Lemma~\ref{lat_lem} to $h=|\rho_B|^d$ and fix $a>0$ such
that
\begin{align}\label{lat_Ka}
 K_a:=\sum_{k\in\Z}|\rho_B(ak)|^d<\infty.
\end{align}
For $Y^{(a)}(s)=B_{as}$, stationarity gives
\begin{align*}
 &\sup_{u\in[0,1]}\sup_{i\ge0}
 \sum_{j=0}^\infty
 \big|\E[Y^{(a)}(1+u+i)Y^{(a)}(1+u+j)]\big|^d
 =\sup_{i\ge0}\sum_{j=0}^\infty
 |\rho_B(a(j-i))|^d\le K_a.
\end{align*}
Thus Proposition~\ref{ns_int} applies to the rescaled process.
For a function   $g\in L^2(\ga)$ with Hermite rank at least $d$,
 put
$V_T^g(t)=T^{-\frac{1}{2}}\int_0^{Tt}g(B_x)\,dx$.
Use $F_\eps^{g,Y^{(a)}}$ to denote \eqref{int_proc} for
$Y^{(a)}$.  The change of variables $x=as$ gives the exact identity
\begin{align*}
 V_T^g(t)=\sqrt a\,F_{a/T}^{g,Y^{(a)}}(t).
\end{align*}
For $T\ge a$, the parameter $a/T$ belongs to $(0,1]$.
Applying \eqref{ns_intbd} on $[0,1]$ therefore yields
\begin{align}\label{cont_max}
 \sup_{T\ge a}\E\Big[ \|V_T^g\|_\infty^2 \Big]
 \le C_{a,d,K_a}\|g\|_2^2.
\end{align}
In particular, with $\varphi_M$ as in \eqref{herm_cut},
\begin{align}\label{cont_approx}
 \sup_{T\ge a}\E\Big[   \|V_T^\varphi-V_T^{\varphi_M}\|_\infty^2 \Big]
 \le C_{a,d,K_a}\|\varphi-\varphi_M\|_2^2
 \longrightarrow0.
\end{align}

For completeness, we verify the increment bound for finite
Hermite expansions.  For $0\le s<t\le1$, the same argument as
in \eqref{disc_trunc_l2} gives

\noi
\begin{align*}
 \E \Big[  |V_T^{\varphi_M}(t)-V_T^{\varphi_M}(s)|^2 \Big]
 \le\frac{\|\varphi_M\|_2^2}{T}
       \int_{Ts}^{Tt}\int_{Ts}^{Tt}
       |\rho_B(x-y)|^d\,dx\,dy
 \le J_d\|\varphi\|_2^2(t-s).
\end{align*}
Hypercontractivity on the finitely many chaoses, followed by the
Kolmogorov tightness criterion,
implies
tightness of $(V^{\varphi_M}_T)_{T\geq a}$ in $\mathbf{C}([0,1])$.  The
finite-dimensional continuous Breuer--Major theorem recalled
in the introduction identifies the limit:
\begin{align}\label{cont_poly}
 V_T^{\varphi_M}\LRA\s_M\cW
 \quad\text{in }\mathbf{C}([0,1]).
\end{align}
Here $\s_M^2$ is \eqref{c_var} truncated at order $M$, and
$ |\s_{\varphi,B}^2-\s_M^2|
 \le J_d\|\varphi-\varphi_M\|_2^2\longrightarrow0.$
The series in \eqref{c_var} is therefore absolutely convergent
and nonnegative, since each $\s_M^2$ is a limiting variance.
Billingsley's theorem \cite[Theorem~3.2]{Bil99}, applied to
\eqref{cont_approx} and \eqref{cont_poly}, completes the proof.
\end{proof}

\section{FCLT for self-similar Gaussian processes}
\label{selfsim}

We prove Corollary~\ref{ss_thm}, stated in
Section~\ref{ss_model}, together with its critical counterpart,
Theorem~\ref{ss_crit}.  The endpoint $p=2$ of
Theorem~\ref{ns_max} and its integral version,
Proposition~\ref{ns_int}, provide uniform path-norm approximation
for these nonstationary processes.
Lemma~\ref{ss_env} verifies the covariance row bound using the
estimates of \cite{CNN20}.  Uniform path-norm approximation by finite
Hermite expansions
then treats a general $\varphi\in L^2(\ga)$.  In the critical regime, the
leading
Hermite chaos gives the limiting variance and the maximal
inequality controls the higher-chaos remainder.

We first justify the definitions and the rescaling
\eqref{ss_inc}--\eqref{ss_law}.  By the reverse triangle inequality
and self-similarity,
\begin{align*}
 \|\Delta_\eps X(s)\|_2
 \ge
 \|X(1)\|_2\big((s+\eps)^\be-s^\be\big)>0,
\end{align*}
so every coordinate of $Y_\eps$ is well defined and standard normal.
Since $Y_\eps$ is jointly measurable and
$\varphi\in L^1(\ga)$, Fubini's theorem implies that
$s\mapsto\varphi(Y_\eps(s))$ is locally integrable almost surely;
hence the corresponding integral paths are continuous.

Self-similarity gives equality of the finite-dimensional
distributions of $\{Y_\eps(s):s\ge0\}$ and $\{Y(s/\eps):s\ge0\}$.
For bounded measurable $\varphi$, Fubini's theorem therefore gives
equality
of all joint moments of finite vectors of the corresponding time
integrals.  Since these vectors are bounded, their laws are determined
by their moments.  Approximating a general
$\varphi\in L^1(\ga)$ by bounded truncations extends the identity in
law to the present setting.  Finally, the almost sure continuity of
the integral paths yields \eqref{ss_law} as an identity of laws on
$\mathbf C([0,1])$.  The approximation thus covers every measurable
$\varphi\in L^1(\ga)$.

\subsection{Covariance estimates and the subcritical proof}
\label{ss_proof}

We now verify \eqref{ns_fib} for the normalized increments in
\eqref{ss_inc}.  The estimates away from the origin are sufficient. Recall that
the initial unit interval was treated separately in
Proposition~\ref{ns_int}.

\begin{lemma}[\textsf{Uniform covariance envelope}]\label{ss_env}
Under the assumptions  \textup{\bf(H.1)} and \textup{\bf(H.2)}  in  Section~\ref{ss_model}, write
$\Phi(s,t)=\E[Y(s)Y(t)]$ and set
\begin{align}\label{kappa_def}
 \kappa=
 \begin{cases}
  \min\{2-\al,\nu\},&0<\al<1,\\
  2-\al,&1\le\al<2.
 \end{cases}
\end{align}
Then,  there exists $C<\infty$ such that
\begin{align}\label{ss_envelope}
 |\Phi(s,t)|\le C(1+|t-s|)^{-\kappa},
 \quad s,t\ge1.
\end{align}
In particular, if $\kappa d>1$, then \eqref{ns_fib} holds with a constant
depending only on $X$ and $d$.  The condition $\al<2- \frac{1}{d}$ implies
$\kappa d>1$.
\end{lemma}

\begin{proof}
Put $b=\be-\frac{\al}{2}\ge0$ and
$v(s)=\E[(\Delta_1X(s))^2]$.
Under \textup{\bf(H.1)}, the increment-variance estimate of
\cite[Lemma~4.1]{CNN20} gives
\begin{align}\label{ss_v}
 \begin{aligned}
 &v(s)=2\ld s^{2b}(1+u_1(s)),
   \quad |u_1(s)|\le Cs^{-\tau},\quad s\ge1,\\
 &\tau=1-\al\ \text{if }\al<1,
   \quad \tau=2-\al\ \text{if }\al\ge1.
 \end{aligned}
\end{align}
The covariance assumptions ensure continuity of $v$ on $[1,\infty)$.
Its positivity, proved above, and \eqref{ss_v} imply
\begin{align}\label{ss_vcomp}
 C^{-1}s^{2b}\le v(s)\le Cs^{2b},\quad s\ge1.
\end{align}
Thus \eqref{ss_vcomp} controls the normalization factors.

By symmetry, take $t=s+h$, with $s\ge1$ and $h\ge0$.
The case where $0\le h\le1$ follows  immediately 
from $|\Phi|\le1$.

Now let us consider the range $h > 1$.
Let $c>1$ be as in {\bf(H.2)}. We further split the range into two cases. 
\begin{itemize}
\item
For $h\ge1$ with $h\le(c-1)s+c$, the comparable-time estimate in
\cite[Lemma~4.2]{CNN20} reads
\begin{align}\label{ss_u}
 \begin{aligned}
 &\E[\Delta_1X(s)\Delta_1X(s+h)]
   =\ld s^{2b}\big(2a_\al(h)+u_2(s,s+h)\big),\\
 &|u_2(s,s+h)|
   \le C\big(s^{-1}h^{\al-1}+s^{\al-2}\big).
 \end{aligned}
\end{align}
Since $|a_\al(h)|\le Ch^{\al-2}$ for $h\ge1$, division by
$\sqrt{v(s)v(s+h)}$ leads us to 

\noi
\begin{align}\label{ss_near}
 |\Phi(s,s+h)|
 \le C\left(\frac{s}{s+h}\right)^b
   \big(h^{\al-2}+s^{-1}h^{\al-1}+s^{\al-2}\big)
 \le Ch^{\al-2}.
\end{align}
For the last inequality, we used $h\le(2c-1)s$, $b\ge0$, and $\al<2$.

\item  For $h\ge(c-1)s+c$, \cite[Lemma~4.3]{CNN20} gives the following
bounds:

\noi
\begin{align}\label{ss_farraw}
 |\E[\Delta_1X(s)\Delta_1X(s+h)]|
 \le C
 \begin{cases}
  s^{2\be+\nu-2}h^{-\nu},&\al<1,\\
  s^{2\be-\al}h^{\al-2},&\al\ge1.
 \end{cases}
\end{align}
If $\al\ge1$, \eqref{ss_vcomp} gives $|\Phi(s,s+h)|\le Ch^{\al-2}$.
If $\al<1$, it gives instead
\begin{align}\label{ss_far}
 |\Phi(s,s+h)|
 \le C\left(\frac{s}{s+h}\right)^b
          s^{\al+\nu-2}h^{-\nu}.
\end{align}
Since $b\ge0$, we have
$(s/(s+h))^b\le1$.
If $\al+\nu-2\ge0$, the   condition
$h\ge(c-1)s+c$, with $c>1$, implies
$s\le h/(c-1)$.  Raising this inequality to the
nonnegative power $\al+\nu-2$ and using
\eqref{ss_far}, we obtain
\begin{align*}
 |\Phi(s,s+h)|
 \le C s^{\al+\nu-2}h^{-\nu}
 &\le C(c-1)^{-(\al+\nu-2)}
        h^{\al+\nu-2}h^{-\nu}\\
 &\le C h^{\al-2},
\end{align*}
where the fixed factor involving $c$ is absorbed
into the constant.
If $\al+\nu-2<0$, we instead use $s\ge1$,
which gives $s^{\al+\nu-2}\le1$.  Thus
\eqref{ss_far} yields
$ |\Phi(s,s+h)|\le C h^{-\nu}.$
The first case corresponds to $\nu\ge2-\al$,
and the second to $\nu<2-\al$.
Hence, for $\al<1$, both cases give us
$
 |\Phi(s,s+h)|
 \le C h^{-\min\{2-\al,\nu\}}.
$
Together with the preceding bound for $\al\ge1$
and  \eqref{ss_near}, 
we obtain the bound   \eqref{ss_envelope} for $h\geq 1$.
\end{itemize}
Combining this with the case $h\leq 1$, we have  \eqref{ss_envelope}
for all $s,t\geq 1$.

For the lattice fiber at $u$, we can deduce from  \eqref{ss_envelope} 
that

\noi
\begin{align*}
 \sup_{i\ge0}\sum_{j\ge0}
 |\Phi(1+u+i,1+u+j)|^d
 \le C\sum_{k\in\Z}(1+|k|)^{-\kappa d},
\end{align*}
uniformly in $u\in[0,1]$.  The series is finite when $\kappa d>1$.
If $\al<1$, both $2-\al$ and $\nu$ exceed one.  If $\al\ge1$, the
condition $\kappa d>1$ is exactly $\al<2- \frac{1}{d}$.
This completes the proof.
\end{proof}

\begin{proof}[Proof of Corollary~\ref{ss_thm}]
By \eqref{ss_law}, it is enough to consider $F_\eps^\varphi$.
For $M\ge d$, let $\varphi_M=\sum_{q=d}^M c_qH_q$.
The finite-dimensional theorem in
\cite[Theorem~1.2\textup{(1)}]{CNN20} applies to $\varphi_M$.
We briefly verify tightness for these finite expansions to identify
precisely the input used in the approximation.

We first remove the initial unit interval by setting
\begin{align*}
 U_\eps^M(t)=\sqrt\eps
 \int_1^{\max\{1,\frac{t}{\eps}\}}\varphi_M(Y(s))\,ds.
\end{align*}
Its $L^2$ uniform distance from $F_\eps^{\varphi_M}$ is at most
$\sqrt\eps\,\|\varphi_M\|_2$, by \eqref{cell_l2}.
For $s,t\ge1$, since the Hermite rank of $\varphi_M$ is at least $d$,
we have 

\noi
\begin{align*}
 |\E[\varphi_M(Y(s))\varphi_M(Y(t))]|
 \le\|\varphi_M\|_2^2|\Phi(s,t)|^d.
\end{align*}
Since $\kappa d>1$, Lemma~\ref{ss_env} makes the right-hand side
integrable in $t$ uniformly in $s$.  Thus
$\| U_\eps^M(t)-U_\eps^M(s)\|^2_2\le C|t-s|$ on $[0,1]$.
Each increment belongs to the sum of Gaussian chaoses of orders
$d,\ldots,M$. 
Therefore, by  hypercontractivity
\cite[Theorem 2.8.12]{NP12}, 
we have for any fixed $p>2$:

\noi
\begin{align*}
 \|U_\eps^M(t)-U_\eps^M(s)\|_p
 \le C_{p,M}\sqrt{|t-s|}.
\end{align*}
The Kolmogorov tightness criterion applies since $U_\eps^M(0)=0$.
Combining this with the finite-dimensional limit yields
\begin{align}\label{ss_poly}
 F_\eps^{\varphi_M}\LRA\s_{\varphi_M,\al}\cW
 \quad\text{in }\mathbf{C}([0,1]).
\end{align}

The difference $\varphi-\varphi_M$ has vanishing Hermite coefficients
below order $d$.  Then, it follows from Lemma~\ref{ss_env} and Proposition~\ref{ns_int} 
that

\noi
\begin{align}\label{ss_approx}
 \sup_{0<\eps\le1}
 \E \big[ \|F_\eps^\varphi-F_\eps^{\varphi_M}\|_\infty^2 \big]
 \le C_{X,d}\|\varphi-\varphi_M\|_2^2\longrightarrow0.
\end{align}
The constant does not depend on $M$.
Finally, $|a_\al|\le1$ and $a_\al\in L^d(\R)$ imply
\begin{align*}
 |\s_{\varphi,\al}^2-\s_{\varphi_M,\al}^2|
 \le\|\varphi-\varphi_M\|_2^2
       \int_\R|a_\al(h)|^d\,dh\longrightarrow0.
\end{align*}

\noi
Thus, the desired functional CLT \eqref{ss_fclt}
follows from Billingsley's theorem  (\cite[Theorem~3.2]{Bil99})
together with
\eqref{ss_poly} and \eqref{ss_approx}.
\qedhere

\end{proof}

\subsection{The critical logarithmic normalization}
\label{ss_boundary}

At $\al=2- \frac{1}{d}$, the leading Hermite chaos is critical, whereas
every
higher chaos remains subcritical.  We first compute the critical
covariance directly.  

\begin{lemma}[\textsf{Critical covariance comparison}]\label{ss_critcov}
Under the assumptions of Section~\ref{ss_model}, suppose that
$\al=2- \frac{1}{d}$.  For $L\ge2$,
\begin{align}\label{crit_compare}
 \int_0^L\int_0^L
 |\Phi(s,t)^d-a_\al(t-s)^d|\,ds\,dt\le CL.
\end{align}
Consequently, for $u,v\ge0$,
\begin{align}\label{crit_covlim}
 \frac{1}{L\log L}\int_0^{Lu}\int_0^{Lv}\Phi(s,t)^d\,ds\,dt
 \longrightarrow
 2A_d^d(u\wedge v),
\end{align}
where $A_d=\frac{(2d-1)(d-1)}{2d^2}$.
When $d=1$, the right-hand side is zero.
\end{lemma}

\begin{proof}
Set $\kappa=2-\al=1/d$ and $b=\be- \frac{\al}{2}$.
By symmetry, it suffices to consider $0\le s\le t\le L$.
The regions where $s<1$ or $t-s<1$ contribute at most
$O(L)$.
We therefore restrict to $s\ge1$ and $h=t-s\ge1$.

Let $c_0=\min\{1,c-1\}>0$
with  $c > 1$ as in \textup{\bf(H.2)}.
First suppose that $1\le h\le c_0s$.
From  \eqref{ss_u} 
and \eqref{ss_v}, we can write
\begin{align*}
 \Phi(s,s+h)
 &=\frac{\E[\Delta_1X(s)\Delta_1X(s+h)]}
         {\sqrt{v(s)v(s+h)}}
         =D(s,h)\Big(a_\al(h)+\frac{u_2(s,s+h)}2\Big),
\end{align*}
where
\begin{align*}
 D(s,h)
 =\Big(\frac{s}{s+h}\Big)^b
  \big((1+u_1(s))(1+u_1(s+h))\big)^{-1/2}.
\end{align*}
By \eqref{ss_vcomp}, $1+u_1(t)$ is uniformly bounded
above and away from zero for $t\ge1$.
Thus $D$ is bounded.
Write $D(s,h)=A(s,h)B(s,h)$, where
\begin{align*}
 A(s,h)=(1+h/s)^{-b}
 \AND
 B(s,h)=\big((1+u_1(s))(1+u_1(s+h))\big)^{-1/2}.
\end{align*}

\noi
Since $0\le h/s\le1$, the mean value theorem applied
to $x\mapsto(1+x)^{-b}$ gives
$|A(s,h)-1|\le Ch/s$.
Moreover, $1+u_1(t)$ is uniformly bounded above
and away from zero by \eqref{ss_vcomp}.
Applying the mean value theorem to
$(x,y)\mapsto(xy)^{-1/2}$ therefore yields
\begin{align*}
 |B(s,h)-1|
 \le C\big(|u_1(s)|+|u_1(s+h)|\big).
\end{align*}
Using $AB-1=(A-1)B+(B-1)$, the boundedness of $B$,
and $|u_1(t)|\le Ct^{-\kappa}$ from \eqref{ss_v},
we obtain
\begin{align*}
 |D(s,h)-1|
 &\le C\Big(\frac hs+|u_1(s)|+|u_1(s+h)|\Big)
 \le C\Big(\frac hs+s^{-\kappa}\Big).
\end{align*}
Moreover, substituting $\al=2-\kappa$ into
\eqref{ss_u} gives
\begin{align*}
 |u_2(s,s+h)|
 &\le C\left(s^{-1}h^{1-\kappa}+s^{-\kappa}\right)
 =Ch^{-\kappa}
   \Big(\frac hs+\Big(\frac hs\Big)^\kappa\Big).
\end{align*}
Since $|a_\al(h)|\le Ch^{-\kappa}$ and
$h/s\le1$, we obtain, using $0<\kappa\le1$,
\begin{align}\label{kap1}
 |\Phi(s,s+h)-a_\al(h)|
 \le Ch^{-\kappa}
 \Big(\Big(\frac hs\Big)^\kappa+s^{-\kappa}\Big).
\end{align}
The same estimates give
\begin{align}\label{kap2}
|\Phi(s,s+h)|+|a_\al(h)|\le Ch^{-\kappa}.
\end{align}
Then, using  $\kappa d=1$
and
$|x^d-y^d|\le d|x-y|(|x|+|y|)^{d-1}$,
we can deduce from \eqref{kap1}--\eqref{kap2} that

\noi
\begin{align}\label{crit_local}
 |\Phi(s,s+h)^d-a_\al(h)^d|
 \le Ch^{-1}
 \Big(\Big(\frac hs \Big)^\kappa+s^{-\kappa}\Big).
\end{align}

\noi
Since $c_0\le1$, 
we have
$\min\{c_0s,L-s\} \leq s$
and thus by  enlarging  the $h$-integration range to $[1,s]$,
we can bound the contribution of this region by

\noi 
\begin{align*}
 C\int_1^L\left(
 s^{-\kappa}\int_1^s h^{\kappa-1}\,dh
 +s^{-\kappa}\log s\right)\,ds
 \le CL.
\end{align*}

It remains to consider $s\ge1$, $h\ge1$, and
$h>c_0s$.
First note that

\noi
\begin{align*}
 |\Phi(s,s+h)^d-a_\al(h)^d|
 &\le |\Phi(s,s+h)|^d+|a_\al(h)|^d\\
 &\le Ch^{-\kappa d}=Ch^{-1}.
\end{align*}

\noi
For each fixed $h$, the allowed $s$-interval has
length at most $h/c_0$.  Hence this region contributes
at most $O(L)$.
Combining these estimates yields
 \eqref{crit_compare}.

Write
$J(x)=\int_0^x\int_0^x a_\al(t-s)^d\,ds\,dt$.
Here and below, $f(x)\sim g(x)$ means
$f(x)/g(x)\to1$ as $x\to\infty$.
For $d\ge2$, we have $\al=2- \frac1 d>1$.
 Taylor expansion, with cancellation of the odd-order terms, 
 allows us to obatin
 
 \noi
\begin{align*}
 a_\al(h)
 &=\frac{\al(\al-1)}2h^{\al-2}
   +O(h^{\al-4})
 \sim A_dh^{-1/d},
\end{align*}
where $A_d=\al(\al-1)/2 = \frac{(2d-1)(d-1)}{2d^2}>0$.
Consequently,
$a_\al(h)^d\sim A_d^d/h$ as $h\to\infty$.
As a result, we have the following asymptotic 

\noi
\begin{align}\label{crit_stat}
 J(x)=2\int_0^x(x-h)a_\al(h)^d\,dh
 \sim2A_d^d x\log x.
\end{align}

For $d=1$, $\al=1$ and $a_1(h)=(1-|h|)_+$, so $J(x)=O(x)$.
For an even translation kernel, the rectangular integral equals
\begin{align*}
 \int_0^{Lu}\int_0^{Lv}a_\al(t-s)^d\,ds\,dt
 =\frac{1}{2}\big(J(Lu)+J(Lv)-J(L|u-v|)\big).
\end{align*}
Dividing by $L\log L$ and using \eqref{crit_stat} gives the limit in
\eqref{crit_covlim} for $a_\al$.
The absolute error on this rectangle is $O(L)$ by
\eqref{crit_compare} on the containing square.
This completes the proof.
\qedhere

\end{proof}

\begin{proof}[Proof of Theorem~\ref{ss_crit}]
We first prove the fixed-chaos limit, with its variance given by
Lemma~\ref{ss_critcov}, and then remove the higher chaoses using
Proposition~\ref{ns_int}.  Put $L= 1/\eps \ge e$ and define
\begin{align*}
 U_L(t)=\frac{1}{\sqrt{L\log L}}
 \int_1^{\max\{1,Lt\}}H_d(Y(s))\,ds.
\end{align*}

\noi
Note that
removing $[0,1]$ changes the path supremum in $L^2$ by at most
$\frac{\|H_d\|_2}{\sqrt{L\log L}}$, by \eqref{cell_l2}.
Now let us  represent $Y(s)=\mathcal{X}(e_s)$ with $\mathcal{X}$
isonormal over a separable Hilbert space $\mathfrak{H}$,
where
$\|e_s\|_{\mathfrak{H}}=1$ and 
$\langle e_s,e_t\rangle_{\mathfrak{H}}=\Phi(s,t)$.
Let $I_d$ denote the $d$th multiple  integral operator
such that
$\E[I_d(f)I_d(g)]=d!\langle f,g\rangle_{\mathfrak{H}^{\otimes d}}$
for any symmetric  $f, g\in\mathfrak{H}^{\otimes d}$, and
let  $\otimes_r$ denote
the $r$th contraction;
see \cite[Section~2]{CNN20}
and \cite{NP12} for any unexplained notation. 
Then $U_L(t)=I_d(k_{L,t})$, with\footnote{The continuity of $\Phi$ on $[1,\infty)^2$ implies
that $s\mapsto e_s\in\mathfrak H$ is continuous,
since
 $\|e_s-e_t\|_{\mathfrak H}^2
 =2\big(1-\Phi(s,t)\big) \to 0$
as $t\to s$.
We may take $\mathfrak H$ to be the closed linear
span of $\{e_s:s\ge1\}$.
By continuity, the vectors with rational indices
have the same closed linear span, so $\mathfrak H$
is separable.
Moreover, since $\|e_s\|_{\mathfrak H}=1$,
a telescoping expansion of the tensor product gives
 $\|e_s^{\otimes d}-e_t^{\otimes d}\|
 \le d\|e_s-e_t\|_{\mathfrak H}.$
Thus $s\mapsto e_s^{\otimes d}$ is continuous
and has norm one, so its integral over any compact
interval exists as a Bochner integral in
$\mathfrak H^{\odot d}$.
The map $I_d:\mathfrak H^{\odot d}\to L^2(\Omega)$
is bounded and linear, and therefore commutes
with this integral.
Together with
$I_d(e_s^{\otimes d})=H_d(Y(s))$,
this justifies $U_L(t)=I_d(k_{L,t})$.}
\begin{align*}
 k_{L,t}=\frac{1}{\sqrt{L\log L}}
 \int_1^{\max\{1,Lt\}}e_s^{\otimes d}\,ds.
\end{align*}

\noi
By \eqref{OL_pq}, we have 
$ \E[H_d(Y(r))H_d(Y(q))]
 =d!\,\Phi(r,q)^d.$
Then,  for $s,t\in[0,1]$,
\begin{align} \label{OL_app2}
\begin{aligned} 
 \E[U_L(s)U_L(t)]
 &=\frac{d!}{L\log L}
   \int_1^{\max\{1,Ls\}}
   \int_1^{\max\{1,Lt\}}
   \Phi(r,q)^d\,dq\,dr\\
 &=\frac{d!}{L\log L}
   \int_0^{Ls}\int_0^{Lt}
   \Phi(r,q)^d\,dq\,dr
   +O\big((\log L)^{-1}\big)
\end{aligned}
\end{align}
as $L\to+\infty$.
Thus, we can deduce from    \eqref{crit_covlim}
that  $ \E[U_L(s)U_L(t)] \to  2d!A_d^d(s\wedge t).$

For $d\ge2$, we verify the contraction condition for CLT on  a  fixed chaos
\cite{NP05} 
and its multivariate extension in  \cite{PT05},
in the form recalled in \cite[Theorem~2.1]{CNN20}.
Fix $r\in\{1,\ldots,d-1\}$ and write $a=r/d \in(0,1)$.
In the convolution estimates below, function norms are Lebesgue norms
on $\R$, while tensor norms are taken in the relevant tensor powers of $\mathfrak{H}$.
Let $M=L+2$ and define the even kernels
\begin{align*}
 f_M(x)&=(1+|x|)^{-a}\ind_{\{|x|\le M\}}
 \AND
 g_M(x)=(1+|x|)^{-(1-a)}\ind_{\{|x|\le M\}}.
\end{align*}
Write $\mathbf x=(x_1,x_2,x_3,x_4)$ and set
\begin{align*}
 J_M=\int_{[1,M]^4}\prod_{i=1}^2
 f_M(x_{2i-1}-x_{2i})g_M(x_i-x_{i+2})\,d\mathbf x.
\end{align*}
The envelope \eqref{ss_envelope} implies, uniformly for $t\in[0,1]$,
\begin{align}\label{crit_ctr}
 \|k_{L,t}\otimes_r k_{L,t}\|^2
 \le\frac{CJ_M}{L^2(\log L)^2}
 \le\frac{CM\|f_M*g_M\|_2^2}{L^2(\log L)^2}.
\end{align}
For the last inequality, fix $x_1$ and extend the other three integrals
to $\R$; the two convolutions give $\|f_M*g_M\|_2^2$.
Choose $p=\frac{2}{1+a}$ and $q=\frac{2}{2-a}$, so $ \frac{1}{p}+
\frac{1}{q}=\frac{3}{2}$.
Direct integration and Young's inequality yield
\begin{align*}
 \|f_M\|_p\le C M^{\frac{1-a}{2}},
 \quad\|g_M\|_q\le C M^{\frac{a}{2}},
 \AND
 \|f_M*g_M\|_2^2\le CM.
\end{align*}
Thus the squared norm of every contraction in \eqref{crit_ctr} is
$O((\log L)^{-2})$ and tends to zero.
Hence, by the fourth moment theorems \cite[Theorem~2.1]{CNN20},
 together with the covariance limit,
 we obtain the finite-dimensional convergences. 

When $d=1$, the vector is already Gaussian and its covariance tends to
zero, so the same conclusion holds without a contraction argument.

For tightness, let $0\le s\le t\le1$.
The integration interval for $U_L(t)-U_L(s)$ has length
$\ell\le L(t-s)$.  The envelope \eqref{ss_envelope} with $\kappa d=1$
and an application of \eqref{OL_pq} as in \eqref{OL_app2}
 imply 
\begin{align*}
 \E\big[ |U_L(t)-U_L(s)|^2 \big]
 \le\frac{C\ell\log(e+\ell)}{L\log L}
 \le C(t-s).
\end{align*}
Then, by
hypercontractivity applied on the   $d$th chaos,
we get the corresponding
$L^p$ bound by  $C_{p,d}|t-s|^{\frac{1}{2}}$ for every $p>2$.
  Consequently, it follows from the Kolmogorov tightness criterion 
  that 
\begin{align*}
 \frac{F_{1/L}^{H_d}}{\sqrt{\log L}}
 \LRA\sqrt{2d!A_d^d}\,\cW
 \quad\text{in }\mathbf{C}([0,1]).
\end{align*}

Finally, write $\varphi=c_dH_d+r_d$.  The remainder $r_d$ has Hermite rank at least $d+1$.  Since
$\kappa (d+1)=\frac{d+1}{d}>1$, Lemma~\ref{ss_env} supplies the fiber
row-sum bound at rank threshold $d+1$.  Proposition~\ref{ns_int} then gives
\begin{align}\label{crit_rem}
 \E\bigg( \left\|
 \frac{F_\eps^{r_d}}{\sqrt{\log(1/\eps)}}
 \right\|_\infty^2 \bigg)
 \le\frac{C_{X,d}\|r_d\|_2^2}{\log(1/\eps)}
 \xrightarrow{\eps\to 0}0.
\end{align}
Combining the leading-chaos limit with \eqref{crit_rem}, and then using
\eqref{ss_law}, proves the result.
\end{proof}

\begin{remark}[\textsf{Normalization of the critical variance}] \label{rem_var63}
\rm
The factor two in \eqref{crit_var} comes from the two symmetric
triangles in the covariance integral \eqref{crit_stat}.
The resulting constant differs from the one printed in
\cite[equation~(1.24), p.~153]{CNN20}; our normalization follows
directly from Lemma~\ref{ss_critcov} and agrees with the stationary
covariance reduction in the proof of
\cite[Lemma~6.3, p.~175]{CNN20}.
The coefficient is independent of $\be$ and $\ld$, since these
parameters disappear from the leading normalized increment
correlation.
\end{remark}

\appendix
\section{Auxiliary Gaussian estimates}
\label{app:aux}

\begin{proof}[Proof of Lemma~\ref{rearr_lem}]
We first prove part~\textup{(i)}.  The case $q=0$ is immediate, so
assume $q\ge1$.  Since
$
 |h_q(x)|\le C_q(1+|x|^q),
$
it is enough to control $\int_A |x|^{2q}\,\ga(dx)$.  Put
$p=\ga(A)\le\frac{1}{2}$.  If $p=0$, there is nothing to prove;
hence assume $p>0$.  By Fubini and the Gaussian tail
bound,
we have

\noi
\begin{align*}
 \int_A |x|^{2q}\,\ga(dx)
  = \int_\R  \bigg( \int_0^{|x|} 2q  r^{2q-1} dr \bigg) \ind_A(x) \ga(dx)
&= 2q\int_0^\infty r^{2q-1}
 \ga\big(A\cap\{|x|>r\}\big)\,dr\\
 &\le
 2q\int_0^\infty r^{2q-1}
 \min\{p,2e^{-\frac{r^2}{2}}\}\,dr.
\end{align*}
Set $R=\sqrt{2\log(\frac{2}{p})}$,
so that $2e^{-\frac{R^2}{2}}=p$.  Splitting the preceding integral at $R$
gives
\begin{align*}
 \int_A |x|^{2q}\,\ga(dx)
 &\le
 2qp\int_0^R r^{2q-1}\,dr
 +4q\int_R^\infty r^{2q-1}e^{-\frac{r^2}{2}}\,dr\\
 &=pR^{2q}
 +2^{q+1} q   \int_{\frac{R^2}{2}}^\infty u^{q-1}e^{-u}\,du
 \quad\text{(with $u=\frac{r^2}{2}$)} \\
 &\leq pR^{2q} +   C_q e^{-\frac{R^2}{2}}(1+R^{2q-2}),
\end{align*}
where $C_q$ depends only on $q$ and may vary from line to line.
Since $2e^{-\frac{R^2}{2}}=p$, we obtain
\begin{align*}
 \int_A |x|^{2q}\,\ga(dx)
 \le C_q p(1+R^{2q})
 \le C_q p\bigl(1+\log( 1/p)\bigr)^q.
\end{align*}
Together with $h_q(x)^2\le C_q(1+|x|^{2q})$, this proves
\eqref{set_rearr}.

\medskip

We next prove part~\textup{(ii)}.  Again, the case $q=0$ is
immediate.  For $q\ge1$, using
$|h_q(x)|\le C_q(1+|x|^q)$, we have
$ \E[U|h_q(\cZ)|] \leq C_q\E[U]  +C_q\E[U|\cZ|^q].$
For any $R>0$, the assumptions $0\le U\le M$ and $\E[U]\le C_0$
yield
\begin{align*}
 \E[U|\cZ|^q]
 \le
 R^q\E[U]
 +M\E\big[|\cZ|^q\ind_{\{|\cZ|>R\}}\big]
 &\le
 C_0R^q
 +M\|\cZ^q\|_2
   \PP(|\cZ|>R)^{\frac{1}{2}}\\
 &\le
  C_0R^q+C_qM e^{-\frac{R^2}{4}}.
\end{align*}
Take
 $R=2\sqrt{\log M}$,
which is well defined for $M\ge2$.
Then,
 $M e^{-\frac{R^2}{4}}=1$
and thus,
 \eqref{wt_rearr} follows immediately.
\end{proof}

\end{document}